\documentclass[a4paper,12pt]{article}

\usepackage[left=2cm,right=2cm, top=2cm,bottom=3cm,bindingoffset=0cm]{geometry}

\usepackage{verbatim}
\usepackage{amsmath}
\usepackage{amsthm}
\usepackage{amssymb}
\usepackage{delarray}
\usepackage{cite}
\usepackage{hyperref}
\usepackage{mathrsfs}
\usepackage{tikz}
\usetikzlibrary{patterns}
\usepackage{caption}
\DeclareCaptionLabelSeparator{dot}{. }
\newcommand{\al}{\alpha}
\newcommand{\be}{\beta}
\newcommand{\ga}{\gamma}
\newcommand{\de}{\delta}
\newcommand{\la}{\lambda}
\newcommand{\om}{\omega}

\newcommand{\eps}{\varepsilon}
\newcommand{\vv}{\varphi}

\theoremstyle{plain}

\numberwithin{equation}{section}

\newtheorem{thm}{Theorem}[section]
\newtheorem{lem}[thm]{Lemma}
\newtheorem{prop}[thm]{Proposition}
\newtheorem{cor}[thm]{Corollary}

\theoremstyle{definition}

\newtheorem{ip}[thm]{Inverse Problem}

\theoremstyle{remark}

\newtheorem{remark}[thm]{Remark}

\DeclareMathOperator*{\Res}{Res}

 \allowdisplaybreaks

\begin{document}

\begin{center}
{\Large\bf Necessary and sufficient conditions for solvability of \\[0.2cm] an inverse problem for higher-order differential operators}
\\[0.5cm]
{\bf Natalia P. Bondarenko}
\end{center}

\vspace{0.5cm}

{\bf Abstract.} We consider an inverse spectral problem that consists in the recovery of the differential expression coefficients for higher-order operators with separated boundary conditions from the spectral data (eigenvalues and weight numbers). This paper is focused on the most principal issue of the inverse spectral theory, namely, on the necessary and sufficient conditions for the solvability of the inverse problem. In the framework of the method of the spectral mappings, we consider the linear main equation of the inverse problem and prove the unique solvability of this equation in the self-adjoint case. The main result is obtained for the first-order system of general form and so can be applied to higher-order differential operators with regular and distribution coefficients. From the theorem on the main equation solvability, we deduce the necessary and sufficient conditions on the spectral data for a class of arbitrary order differential operators with distribution coefficients. As a corollary of our general results, we obtain the characterization of the spectral data for the fourth-order differential equation in terms of asymptotics and simple structural properties. 

\medskip


\medskip

  
\vspace{1cm}

\section{Introduction} \label{sec:intr}

This paper is concerned with inverse spectral problems for differential equations of form
\begin{align} \nonumber
\ell_n(y) := & y^{(n)} + \sum_{k = 0}^{\lfloor n/2\rfloor - 1} (\tau_{2k}(x) y^{(k)})^{(k)} \\ \label{eqv} + & \sum_{k = 0}^{\lfloor (n-1)/2\rfloor - 1}  \bigl((\tau_{2k+1}(x) y^{(k)})^{(k+1)} + (\tau_{2k+1}(x) y^{(k+1)})^{(k)}\bigr) = \la y, \: x \in (0,1),
\end{align}
where $n \ge 2$, the notation $\lfloor a \rfloor$ means rounding a real number $a$ down,  the coefficients $\{ \tau_{\nu} \}_{\nu = 0}^{n-2}$ in general can be generalized functions (distributions), the functions $i^{n + \nu} \tau_{\nu}$ are assumed to be real-valued, $\la$ is the spectral parameter. 

We investigate the recovery of the coefficients $\{ \tau_{\nu} \}_{\nu = 0}^{n-2}$ from the eigenvalues $\{ \la_{l,k} \}_{l \ge 1}$ and the weight numbers $\{ \be_{l,k} \}_{l \ge 1}$ of the boundary value problems $\mathcal L_k$, $k = \overline{1,n-1}$, for equation \eqref{eqv} with the separated boundary conditions 
\begin{equation} \label{bc}
y^{[j-1]}(0) = 0, \quad j = \overline{1,k}, \qquad
y^{[s-1]}(1) = 0, \quad s = \overline{1,n-k}.
\end{equation}
Thus, the problem $\mathcal L_k$ has $k$ boundary conditions at $x = 0$ and $(n-k)$ boundary conditions at $x = 1$. The quasi-derivatives $y^{[j]}$ and the weight numbers $\be_{l,k}$ will be rigorously defined in Section~\ref{sec:prelim}.

Our spectral data $\{ \la_{l,k}, \be_{l,k} \}_{l \ge 1, \, k = \overline{1,n-1}}$ generalize the spectral data $\{ \la_l, \be_l \}_{l \ge 1}$ of the classical Sturm-Liouville problem
\begin{gather} \label{StL}
-y'' + q(x) y = \la y, \quad x \in (0, 1), \\ \label{bcStL}
y(0) = y(1) = 0.
\end{gather}
Here $\{ \la_l \}_{l \ge 1}$ are the eigenvalues of the boundary value problem \eqref{StL}--\eqref{bcStL} and $\be_l := \left( \int_0^1 y_l^2(x) \, dx \right)^{-1}$, where $\{ y_l(x) \}_{l \ge 1}$ are the eigenfunctions normalized by the condition $y_l'(0) = 1$. Inverse problems for the second-order Sturm-Liouville equation \eqref{StL} have been studied fairly completely (see the monographs \cite{Mar77, Lev84, PT87, FY01, Krav20} and references therein). In particular, it is well-known that the potential $q(x)$ can be uniquely reconstructed from the spectral data $\{ \la_l, \be_l \}_{l \ge 1}$ by the method of Gelfand and Levitan \cite{GL51}. Nevertheless, this method appeared to be ineffective for higher-order differential equations
\begin{equation} \label{eqvp}
y^{(n)} + \sum_{s = 0}^{n-2} p_s(x) y^{(s)} = \la y, \quad x \in (0,1).
\end{equation}

The inverse problem by the spectral data $\{ \la_{l,k}, \be_{l,k} \}_{l \ge 1, \, k = \overline{1,n-1}}$ for equation \eqref{eqvp}
with the coefficients $p_s \in C^s[0,1]$, $s = \overline{0,n-2}$,
has been introduced by Leibenzon \cite{Leib66}, who has proved the uniqueness of its solution. In \cite{Leib71, Leib72}, Leibenzon developed a constructive method for finding the solution and obtained the necessary and sufficient conditions for the solvability of the inverse problem. However, Yurko \cite{Yur02} has shown that the spectral data of Leibenzon uniquely specify the coefficients $\{ p_s \}_{s = 0}^{n-2}$ only under the so-called \textit{separation condition}: the eigenvalues of any two neighboring problems $\mathcal L_k$ and $\mathcal L_{k+1}$ must be distinct. Furthermore, Yurko introduced another spectral characteristics, which is now called \textit{the Weyl-Yurko matrix}. It generalizes the spectral data $\{ \la_{l,k}, \be_{l,k} \}_{l \ge 1, \, k = \overline{1,n-1}}$ of Leibenzon and uniquely determines the higher-order differential operators in the general case, without any restrictions on their spectra. As a result, Yurko has created the inverse problem theory for equation \eqref{eqvp} with $p_s \in W_2^{s + \nu}$, $s = \overline{0,n-2}$, $\nu \ge 0$, on a finite interval and on the half-line (see \cite{Yur92, Yur00, Yur02}). The inverse scattering problem on the line requires a different approach (see \cite{Beals85, Beals88}). We also mention that inverse spectral problems for higher-order differential operators in other statements were considered in \cite{Bar74, Khach76, McK81, McL86, PK97, CPS98, Glad05, BK15, BK21, PB20} and other studies.

Generalizations of the inverse spectral problems by the Weyl-Yurko matrix and by the spectral data $\{ \la_{l,k}, \be_{l,k} \}_{l \ge 1, \, k = \overline{1,n-1}}$ to the higher-order differential operators with distribution coefficients have been investigated by Bondarenko \cite{Bond21, Bond22-alg, Bond23-mmas, Bond23-reg, Bond23-loc, Bond23-res}. In particular, the papers \cite{Bond21, Bond23-mmas, Bond23-reg} are mostly focused on uniqueness theorems. In \cite{Bond22-alg}, nonlinear inverse problems have been reduced to the linear equation
\begin{equation} \label{maineq}
(I - \tilde R(x)) \psi(x) = \tilde \psi(x),
\end{equation}
in the Banach space $m$ of bounded infinite sequences. Equation \eqref{maineq} is considered for each fixed $x \in [0,1]$. Here $\psi(x), \tilde \psi(x) \in m$, $\tilde R(x)$ is a compact operator and $I$ is the identity operator in $m$. The vector $\tilde \psi(x)$ and the operator $\tilde R(x)$ are constructed by using the given spectral data, while the unknown vector $\psi(x)$ is related to the coefficients of equation \eqref{eqv}. The further details regarding the main equation \eqref{maineq} can be found in Section~\ref{sec:main}. 

For existence of the inverse problem solution, the solvability of equation \eqref{maineq} is crucial. However, this issue is very difficult for investigation. Leibenzon \cite{Leib72} and Yurko \cite{Yur00, Yur02} in their studies imposed the requirement of the existence of a bounded inverse operator $(I - \tilde R(x))^{-1}$. But, in general, it is difficult to verify this requirement. The only relatively simple situation is the case of small $\| R(x) \|$. This case of local solvability was considered in \cite{Yur00, Yur02} for regular coefficients and in \cite{Bond23-loc} for distributions. For $n = 2$, the unique solvability of the main equation \eqref{maineq} can be proved in the self-adjoint case. Recently, it has been proved for $n = 3$ (see \cite{Bond23-res}). Nevertheless, to the best of the author's knowledge, there were no such results for $n > 3$ even in the case of regular coefficients. This paper aims to study the solvability of the main equation \eqref{eqv} in the self-adjoint case for arbitrary even and odd orders $n$ and to obtain necessary and sufficient conditions on the spectral data $\{ \la_{l,k}, \be_{l,k} \}_{l \ge 1\, k = \overline{1,n-1}}$ for equation \eqref{eqv} with distribution coefficients. 

It is worth mentioning that boundary value problems for linear differential equations of form \eqref{eqv} and inverse spectral problems for them appear in various applications. The majority of applications deal with $n = 2$.
Direct and inverse Sturm-Liouville problems arise in classical and quantum mechanics, geophysics, material science, acoustics, and other branches of science and engineering (see, e.g., the references in \cite{FY01, Krav20}). 
The third-order differential equations are applied to describing elastic beam vibrations \cite{Greg87} and thin membrane flow of viscous liquid \cite{BP96}. Linear differential operators of orders $n = 4$ and $n = 6$ arise in geophysics \cite{Bar74} and in vibration theory \cite{Glad05, MZ13}. Furthermore, characterization of the spectral data has a fundamental significance for spectral theory of linear differential operators.

In order to deal with equation \eqref{eqv}, we apply the regularization approach, which has been developed in \cite{MS16} and a number of subsequent studies. Namely, we reduce equation \eqref{eqv} to the first-order system
\begin{equation} \label{sys}
Y'(x) = (F(x) + \Lambda) Y(x), \quad x \in (0, 1),
\end{equation}
where $Y(x)$ is a column-vector function related to $y(x)$, $\Lambda$ is the $(n \times n)$-matrix whose entry at the position $(n,1)$ equals $\la$ and all the other entries equal zero, the matrix function $F(x) = [f_{k,j}(x)]_{k,j = 1}^n$ with integrable entries is the so-called associated matrix, which is constructed by the coefficients $\{ \tau_{\nu} \}_{\nu = 0}^{n-2}$ of the differential expression $\ell_n(y)$ in a special way. 

Constructions of associated matrices for various classes of differential operators were obtained in \cite{EM99, MS16, MS19, Vlad04, Vlad17, VNS21, KMS23, Bond23-reg} and other studies. As an example, consider equation \eqref{eqv} for $n = 2$:
\begin{equation} \label{StL1}
y'' + \tau_0 y = \la y. 
\end{equation}
Suppose that $\tau_0 \in W_2^{-1}[0,1]$, that is, $\tau_0 = \sigma'$, where $\sigma$ is some function of $L_2[0,1]$. Then, equation \eqref{StL1} can be reduced to the $(2 \times 2)$ system \eqref{sys} with the following associated matrix (see \cite{SS03}):
\begin{equation} \label{F2}
F(x) = \begin{bmatrix}
            -\sigma(x) & 1 \\
            -\sigma^2(x) & \sigma(x)
       \end{bmatrix}.
\end{equation}

In this paper, we consider the main equation \eqref{maineq} constructed for the first-order system \eqref{sys} in the general form. We formulate sufficient conditions for the unique solvability of the main equation in Theorem~\ref{thm:inverse}. Our conditions include the asymptotics, some structural properties of the spectral data, and also the self-adjointness requirements. The proof of Theorem~\ref{thm:inverse} is based 
on construction and contour integration of some functions, which are meromorphic in the complex plane of the spectral parameter. The obtained contour integrals, on the one hand, can be estimated as the radii of the contours tend to infinity. On the other hand, they can be calculated by the Residue Theorem. Although the idea of the proof arises from the cases $n = 2$ and $n = 3$, which were considered in \cite{FY01, Yur02, Bond21-tamkang} and \cite{Bond23-res}, respectively, the solvability of the main equation for $n > 3$ is a fundamentally novel result and the proofs require several new constructions. 

Since Theorem~\ref{thm:inverse} is obtained for the system \eqref{sys} in the general form, then it can be applied to different classes of differential operators with distribution as well as integrable coefficients. Nevertheless, it is important to note that in general the matrix function $F(x)$ cannot be uniquely recovered from the spectral data (see Example 1 in \cite{Bond23-mmas}). Therefore, in the next part of this paper, we consider the system \eqref{sys} associated with equation \eqref{eqv}, in which $\tau_{\nu} \in W_2^{\nu-1}[0,1]$ and $i^{n + \mu} \tau_{\nu}$ are real-valued functions for $\nu = \overline{0,n-2}$. For this class of operators, in \cite{Bond23-loc} the coefficients of the differential expression $\ell_n(y)$ have been reconstructed as some series by using the solution $\psi(x)$ of the main equation \eqref{maineq}.
Moreover, the convergence of those series have been proved in the corresponding functional spaces, including the space $W_2^{-1}[0,1]$ of generalized functions. Theorem~\ref{thm:inverse} together with the results of \cite{Bond23-loc} readily imply the solvability conditions of the inverse spectral problem (Theorem~\ref{thm:nsc}). For $n = 2, 3$, our conditions coincide with the previously known results (see \cite{HM03, Bond23-res}). For $n = 4$, we obtain a novel theorem (Theorem~\ref{thm:nsc4}), which completely characterizes the corresponding spectral data in terms of the asymptotics recently derived in \cite{Bond23-asympt4} and structural properties.

The paper is organized as follows. In Section~\ref{sec:prelim}, we define the spectral data $\{\la_{l,k}, \be_{l,k} \}_{l \ge 1, \, k = \overline{1,n-1}}$ for the system \eqref{sys} and provide other preliminaries. In Section~\ref{sec:sa}, the self-adjoint case is considered and some necessary conditions on the spectral data are established. In Section~\ref{sec:main}, we provide the construction of the main equation \eqref{main} from \cite{Bond22-alg} and discuss some useful properties of the functions, which participate in the main equation. Section~\ref{sec:inverse} contains Theorem~\ref{thm:inverse} on the sufficient conditions for the unique solvability of the main equation together with its proof. In Section~\ref{sec:ip}, we apply Theorem~\ref{thm:inverse} to the inverse problem for equation \eqref{eqv} with $\tau_{\nu} \in W_2^{\nu - 1}[0,1]$. Section~\ref{sec:ex} contains examples for $n = 2, 3, 4$.

\smallskip

Throughout the paper, we use the following \textbf{notations}:

\begin{enumerate}
\item $\de_{j,k}$ is the Kronecker delta. 
\item In estimates, the same symbol $C$ is used for various positive constants which do not depend on $x$, $\lambda$, $l$, etc.
\item $F^0(x) \equiv [\de_{k+1,j}]_{k,j = 1}^n$.
\item Along with $F(x)$, we consider matrix functions $\tilde F(x)$, $F^{\star}(x)$, $\tilde F^{\star}(x)$, $F^0(x)$. If a symbol $\gamma$ denotes an object related to $F(x)$, then the notations $\tilde \ga$, $\gamma^{\star}$, $\tilde \ga^{\star}$, $\ga^0$ are used for the similar object related to $\tilde F(x)$, $F^{\star}(x)$, $\tilde F^{\star}(x)$, $F^0(x)$, respectively.
\item The notation $a_{\langle k \rangle}(\la_0)$ is used for the $k$-th coefficient of the Laurent series for a function $a(\la)$ at a point $\la = \la_0$:
$$
a(\la) = \sum_{k = -q}^{\infty} a_{\langle k \rangle}(\la_0) (\la - \la_0)^k.
$$
\end{enumerate}

\section{Preliminaries} \label{sec:prelim}

Denote by $\mathfrak F_n$ the class of $(n \times n)$ matrix functions $F(x) = [f_{k,j}(x)]_{k,j = 1}^n$ satisfying the conditions
$$
f_{k,j} = \de_{k+1,j}, \, k < j, \quad f_{k,k} \in L_2[0,1], \quad 
f_{k,j} \in L_1[0,1], \, k \ge j, \quad \mbox{trace}(F(x)) = 0.
$$
Thus, the structure of $F \in \mathfrak F_n$ can be symbolically presented as follows:
$$
\begin{bmatrix}
L_2 & 1 & 0 & \dots & 0 & 0 \\
L_1 & L_2 & 1 & \dots & 0 & 0 \\
L_1 & L_1 & L_2 & \dots & 0 & 0 \\
\hdotsfor{6} \\
L_1 & L_1 & L_1 & \dots & L_2 & 1 \\
L_1 & L_1 & L_1 & \dots & L_1 & L_2
\end{bmatrix}.
$$

In this section, we consider the system \eqref{sys} with an arbitrary matrix $F \in \mathfrak F_n$ and define the corresponding spectral data. This section is mainly based on the results of \cite{Bond21, Bond22-alg, Bond22-asympt, Bond23-loc}.

\subsection{Eigenvalues}
Let $F \in \mathfrak F_n$.
Using the entries of $F(x)$, define the quasi-derivatives
\begin{equation} \label{quasi}
y^{[0]} := y, \quad y^{[k]} := (y^{[k-1]})' - \sum_{j = 1}^k f_{k,j} y^{[j-1]}, \quad k = \overline{1,n},
\end{equation}
and the domain
\begin{equation} \label{defDF}
\mathcal D_F := \{ y \colon y^{[k]} \in AC[0,1], \, k = \overline{0,n-1} \}.
\end{equation}

For $y \in \mathcal D_F$, define the column vector function $\vec y(x) = [y^{[j-1]}(x)]_{j = 1}^n$. Obviously, if $y \in \mathcal D_F$, then 
$y^{[n]} \in L_1[0,1]$ and the equation 
\begin{equation} \label{eqvn}
y^{[n]} = \la y, \quad x \in (0,1), 
\end{equation}
is equivalent to the system \eqref{sys} with $Y(x) = \vec y(x)$. Indeed, the first $(n-1)$ rows of \eqref{sys} correspond to the definitions \eqref{quasi}  of the quasi-derivatives and the $n$-th row, to equation \eqref{eqvn}. Below, we consider solutions of \eqref{eqvn} belonging to the domain $\mathcal D_F$.

For $k = \overline{1,n}$, denote by $\mathcal C_k(x, \la)$ the solution of equation \eqref{eqvn} satisfying the initial conditions
\begin{equation} \label{initC}
\mathcal C_k^{[j-1]}(0, \la) = \de_{k,j}, \quad j = \overline{1,n}.
\end{equation}
Clearly, the matrix function $\mathcal C(x, \la) := [\mathcal C_k^{[j-1]}(x, \la)]_{j,k = 1}^n$ is a fundamental solution of the first-order system \eqref{sys}. Therefore, the solutions $\{ \mathcal C_k(x, \la) \}_{k = 1}^n$ exist and are unique. Moreover, their quasi-derivatives $\mathcal C_k^{[j-1]}(x, \la)$ are entire functions of $\la$ for each fixed $x \in [0,1]$ and $k,j = \overline{1,n}$.

For $k = \overline{1,n-1}$, denote by $\mathcal L_k$ the boundary value problem for equation \eqref{eqvn} with the boundary conditions \eqref{bc}.
It can be shown in the standard way that, for each $k = \overline{1,n-1}$, the problem $\mathcal L_k$ has a countable set of eigenvalues $\{ \la_{l,k} \}_{l \ge 1}$, which coincide with the zeros of the entire characteristic function
$$
\Delta_{k,k}(\la) := \det\left( [\mathcal C_r^{[n - j]}(1,\la)]_{j,r = k+1}^n\right).
$$

Furthermore, the eigenvalues satisfy the following asymptotic relations (see \cite{Bond22-asympt}):
\begin{equation} \label{asymptla}
\la_{l,k} = (-1)^{n-k} \left( \frac{\pi}{\sin \frac{\pi k}{n}} (l + \chi_k + \varkappa_{l,k})\right)^n, \quad l \ge 1, \, k = \overline{1,n-1},
\end{equation}
where $\{ \varkappa_{l,k} \} \in l_2$ and $\{ \chi_k \}_{k = 1}^{n-1}$ are constants which do not depend on the matrix $F(x)$. Hence, the constants $\{ \chi_k \}_{k = 1}^{n-1}$ can be determined by the eigenvalues $\{ \la_{l,k}^0 \}_{l \ge 1}$ of the matrix function $F^0(x)$ with the zero entries $f_{s,j}^0(x) \equiv 0$ for $s \ge j$:
$$
\chi_k = \lim_{l \to \infty} \left(\tfrac{1}{\pi} \sin (\tfrac{\pi k}{n}) \sqrt[n]{(-1)^{n-k}\la_{l,k}^0} - l \right), \quad k = \overline{1,n-1}.
$$

\subsection{Weyl solutions and Weyl-Yurko matrix}

For $k = \overline{1,n}$, denote by $\Phi_k(x, \la)$ the solution of equation \eqref{eqvn} satisfying the boundary conditions
\begin{equation} \label{bcPhi}
\Phi_k^{[j-1]}(0,\la) = \de_{k,j}, \quad j = \overline{1,k}, \qquad \Phi_k^{[s-1]}(1,\la) = 0, \quad s = \overline{1, n-k}.
\end{equation}

The functions $\{ \Phi_k(x, \la) \}_{k = 1}^n$ are called \textit{the Weyl solutions} of equation \eqref{eqvn}. Let us summarize the properties of the Weyl solutions, which have been discussed in \cite{Yur02} for the case of higher-order differential operators with regular coefficients and in \cite{Bond21, Bond22-alg} for the system \eqref{sys} in more detail.
For each fixed $x \in [0,1]$ and $k = \overline{1,n}$, the quasi-derivatives $\Phi_k^{[j]}(x, \la)$, $j = \overline{0,n-1}$, are meromorphic in the $\la$-plane and have poles at the eigenvalues $\{ \la_{l,k} \}_{l \ge 1}$. Furthermore, the Weyl solutions are ranked by their growth as $|\la| \to \infty$. In order to estimate them, put $\la = \rho^n$ and divide the $\rho$-plane into the sectors
$$
\Gamma_s := \left\{ \rho \in \mathbb C \colon \frac{\pi(s-1)}{n} < \arg \rho < \frac{\pi s}{n}\right\}, \quad s = \overline{1,2n}.
$$

In each fixed sector $\Gamma_s$, denote by $\{ \om_k \}_{k = 1}^n$ the roots of the equation $\om^n = 1$ numbered so that
\begin{equation} \label{order}
\mbox{Re} \, (\rho \om_1) < \mbox{Re} \, (\rho \om_2) < \dots < \mbox{Re} \, (\rho \om_n), \quad \rho \in \Gamma_s.
\end{equation}

Consider the matrix $F^0(x) = [\de_{k + 1, j}]_{k,j = 1}^n$ and the corresponding eigenvalues $\{ \la_{l,k}^0 \}_{l \ge 1}$ of the boundary value problems $\mathcal L_k^0$, $k = \overline{1,n-1}$. For a fixed sector $\Gamma_s$, denote $\rho_{l,k}^0 = \sqrt[n]{\la_{l,k}^0} \in \overline{\Gamma_s}$.  It can be shown that, for sufficiently large $l$, the eigenvalues $\la_{l,k}^0$ are real (see Lemma~\ref{lem:0}), so $\rho_{l,k}^0$ lie on the boundary of the sector $\Gamma_s$.
Introduce the region
\begin{equation} \label{defGs}
\Gamma_{s,\rho^*,\de} := \{ \rho \in \Gamma_s \colon |\rho| > \rho^*, \, |\rho - \rho_{l,k}^0| > \de, \, l \ge 1, \, k = \overline{1,n-1} \},
\end{equation}
for some positive numbers $\rho^*$ and $\de$.

\begin{prop}[\cite{Bond21, Bond22-alg}] \label{prop:estPhi}
The following estimate holds:
$$
|\Phi_k^{[j-1]}(x, \rho^n)| \le C |\rho|^{j-k} |\exp(\rho \om_k x)|, \quad k,j = \overline{1,n}, \quad x \in [0,1], \quad \rho \in \overline{\Gamma}_{s,\rho^*,\de}
$$
for each fixed $s = \overline{1,m}$, a sufficiently small $\de > 0$, and some $\rho^* > 0$. The numbers $\{ \om_k \}_{k = 1}^m$ are supposed to be numbered in the order \eqref{order} associated with the sector $\Gamma_s$.
\end{prop}

Clearly, the matrix function $\Phi(x, \la) := [\Phi_k^{[j-1]}(x, \la)]_{j,k = 1}^n$ is a solution of the system \eqref{sys}. Therefore, $\Phi(x, \la)$ is related to the fundamental solution $\mathcal C(x, \la)$ as follows:
\begin{equation} \label{defM}
\Phi(x, \la) = \mathcal C(x, \la) M(\la),
\end{equation}
where $M(\la) = [M_{j,k}(\la)]_{j,k = 1}^n$ is some matrix function called \textit{the Weyl-Yurko matrix}. The Weyl-Yurko matrix for the first time was introduced by Yurko \cite{Yur92, Yur00, Yur02} for investigation of inverse spectral problems for higher-order differential operators with regular coefficients. 

It follows from \eqref{initC}, \eqref{bcPhi}, and \eqref{defM} that $M(\la)$ is a unit lower-triangular matrix:
$$
M(\la) = \begin{bmatrix}
            1 & 0 & \dots & 0 \\
            M_{2,1}(\la) & 1 & \dots & 0 \\
            \hdotsfor{4} \\
            M_{n,1}(\la) & M_{n,2}(\la) & \dots & 1
\end{bmatrix}.   
$$

Furthermore, the entries $M_{j,k}(\la)$ for $j > k$ are meromorphic in $\la$ with poles at the eigenvalues $\{ \la_{l,k} \}_{l \ge 1}$. 
In other words, the poles of the $k$-th column coincide with the zeros of the corresponding characteristic function $\Delta_{k,k}(\la)$.

\subsection{Weight matrices and weight numbers}

Denote by $\mathfrak F_{n,simp}$ a class of matrix functions $F(x)$ of $\mathfrak F_n$ such that the corresponding eigenvalues $\{ \la_{l,k} \}_{l \ge 1, \, k = \overline{1,n-1}}$ fulfill the following simplifying assumptions:

\medskip

\textbf{(A-1)} For each fixed $k = \overline{1,n-1}$, the eigenvalues $\{ \la_{l,k} \}_{l \ge 1}$ are simple.

\smallskip

\textbf{(A-2)} $\{ \la_{l,k} \}_{l \ge 1} \cap \{ \la_{l,k+1} \}_{l \ge 1} = \varnothing$ for $k = \overline{1,n-1}$.

\medskip

Note that, in view of the asymptotics \eqref{asymptla}, the conditions (A-1) and (A-2) hold for all sufficiently large indices $l$.
In terms of the Weyl-Yurko matrix, the assumptions (A-1) and (A-2) mean that all the poles of $M(\la)$ are simple and neighbouring columns do not have common poles, respectively. Hence, under the assumption (A-1), the Laurent series of $M(\la)$ at each pole $\la  = \la_{l,k}$ has the form
$$
M(\la) = \frac{M_{\langle -1 \rangle}(\la_{l,k})}{\la - \la_{l,k}} + M_{\langle 0 \rangle}(\la_{l,k}) + M_{\langle 1 \rangle}(\la_{l,k})(\la - \la_{l,k}) + \dots,
$$
where $M_{\langle j \rangle}(\la_{l,k})$ are $(n \times n)$ matrix coefficients. Define the weight matrices
\begin{equation} \label{defN}
\mathcal N(\la_{l,k}) := (M_{\langle 0 \rangle}(\la_{l,k}))^{-1} M_{\langle -1 \rangle}(\la_{l,k}), \quad
\mathcal N(\la_{l,k}) = [\mathcal N_{j,r}(\la_{l,k})]_{j,r = 1}^n.
\end{equation}

Due to \cite[Lemma 4]{Bond22-alg}, under the assumption (A-2), the weight matrices have the following structure:
\begin{equation} \label{structN}
\mathcal N_{j,r}(\la_{l,k}) \ne 0 \quad \Leftrightarrow \quad j = r+1, \: \Delta_{r,r}(\la_{l,k}) = 0.
\end{equation}
Thus, in this case, the weight matrices $\{ \mathcal N(\la_{l,k}) \}_{l \ge 1, \, k = \overline{1,n-1}}$ are uniquely specified by the \textit{weight numbers}
\begin{equation} \label{betaM}
\be_{l,k} := \mathcal N_{k+1,k}(\la_{l,k}) =  \Res_{\la = \la_{l,k}} M_{k+1,k}(\la), \quad l \ge 1, \, k = \overline{1,n-1}.
\end{equation}
In particular, \eqref{structN} implies that $\be_{l,k} \ne 0$. Using the method of \cite{Bond22-asympt, Bond23-asympt4}, the following asymptotics of the weight numbers can be obtained:
\begin{equation} \label{asymptbe}
\be_{l,k} = -n \la_{l,k} (1 + \varkappa_{l,k}^0), \quad \{ \varkappa_{l,k}^0 \} \in l_2, \quad l \ge 1, \, k = \overline{1,n-1}.
\end{equation}

The assumptions (A-1) and (A-2) were imposed in a number of previous studies \cite{Leib66, Leib71, Leib72, Yur00, Yur02, Bond23-loc}. Under these assumptions, coefficients of the differential expression \eqref{eqv} are typically uniquely specified by the spectral data $\{ \la_{l,k}, \be_{l,k} \}_{l \ge 1, \, k = \overline{1,n-1}}$. Uniqueness theorems of this kind have been proved in \cite{Leib66, Yur00, Yur02, Bond23-loc} for various classes of regular and distributional coefficients. However, the matrix function $F(x)$ in general is not uniquely determined by the spectral data (see Example 1 in \cite{Bond23-mmas}).

\subsection{Matrix $F^{\star}(x)$}

Along with $F(x) \in \mathfrak F_n$, consider the matrix function
\begin{equation} \label{defFs}
F^{\star}(x) = [f_{k,j}^{\star}(x)]_{k,j = 1}^n, \quad f_{k,j}(x) = (-1)^{k+j+1} f_{n-j+1,n-k+1}(x).
\end{equation}

Obviously, $F^{\star}(x)$ also belongs to $\mathfrak F_n$. Using the entries of $F^{\star}(x)$, define the quasi-derivatives
\begin{equation} \label{quasiz}
z^{[0]} := z, \quad z^{[k]} := (z^{[k-1]})' - \sum_{j = 1}^k f^{\star}_{k,j} z^{[j-1]}, \quad k = \overline{1,n},
\end{equation}
the domain
$$
\mathcal D_{F^{\star}} := \{ z \colon z^{[k]} \in AC[0,1], \, k = \overline{0,n-1} \},
$$
and consider the equation
\begin{equation} \label{eqvz}
(-1)^n z^{[n]} = \mu z, \quad x \in (0, 1),
\end{equation}
analogous to \eqref{eqvn}. Equation \eqref{eqvz} is equivalent to the first-order system
$$
\frac{d}{dx}\vec z(x) = (F^{\star}(x) + (-1)^n\Lambda) \vec z(x), \quad x \in (0,1),
$$
where $\vec z = [z^{[j-1]}]_{j = 1}^n$ is a column vector and the quasi-derivatives are understood in the sense \eqref{quasiz}.
We agree that, for $y \in \mathcal D_F$, the quasi-derivatives are defined by \eqref{quasi} and, for $z \in \mathcal D_{F^{\star}}$, by \eqref{quasiz}. The solutions of \eqref{eqvn} and \eqref{eqvz} are considered in the domains $\mathcal D_F$ and $\mathcal D_{F^{\star}}$, respectively. 

For $y \in \mathcal D_F$ and $z \in \mathcal D_{F^{\star}}$, define the Lagrange bracket:
\begin{equation} \label{defLagr}
\langle z, y \rangle = \sum_{k = 0}^{n-1} (-1)^k z^{[k]} y^{[n-k-1]}.
\end{equation}

If $y$ and $z$ satisfy equations \eqref{eqvn} and \eqref{eqvz}, respectively, then
\begin{equation} \label{wron}
\frac{d}{dx} \langle z, y \rangle = (\la - \mu) zy,
\end{equation}
see \cite[Section 2.1]{Bond22-alg}.

Analogously to $\{ \mathcal C_k(x, \la) \}_{k = 1}^n$ and $\{ \Phi_k(x, \la) \}_{k = 1}^n$, we define the solutions $\{ \mathcal C_k^{\star}(x, \la) \}_{k = 1}^n$ and $\{ \Phi_k^{\star}(x, \la) \}_{k = 1}^n$ of equation \eqref{eqvz} satisfying the initial conditions \eqref{initC} and the boundary conditions \eqref{bcPhi}, respectively. Furthermore, put $\mathcal C^{\star}(x, \la) := [\mathcal C^{\star[j-1]}_k(x, \la)]_{j,k = 1}^n$, $\Phi^{\star}(x, \la) := [\Phi^{\star[j-1]}_k(x, \la)]_{j,k = 1}^n$, and
define the Weyl-Yurko matrix $M^{\star}(\la) = [M^{\star}_{j,k}(\la)]_{j,k = 1}^n$ by the relation
\begin{equation} \label{defM*}
\Phi^{\star}(x, \la) = \mathcal C^{\star}(x, \la) M^{\star}(\la).
\end{equation}

\begin{prop}[\cite{Bond22-alg}] \label{prop:M}
The Weyl-Yurko matrices $M(\la)$ and $M^{\star}(\la)$ are related as follows:
\begin{equation} \label{relM}
(M^{\star}(\la))^T = J (M(\la))^{-1} J^{-1}, 
\end{equation}
where $J = [(-1)^j\de_{j,n-k+1}]_{j,k = 1}^n$ and $T$ denotes the matrix transpose.
\end{prop}

Consider the boundary value problems $\mathcal L_k^{\star}$, $k = \overline{1,n-1}$, for equation \eqref{eqvz} with the boundary conditions \eqref{bc}, where the quasi-derivatives are defined by \eqref{quasiz}. Define the spectral data $\{ \la_{l,k}^{\star}, \be_{l,k}^{\star} \}_{l \ge 1, \, k = \overline{1,n-1}}$ analogously to $\{ \la_{l,k}, \be_{l,k} \}_{l \ge 1, \, k = \overline{1,n-1}}$.

\begin{lem} \label{lem:simps}
Suppose that $F \in \mathfrak F_{n,simp}$. Then $F^{\star} \in \mathfrak F_{n,simp}$ and
$$
\la_{l,k}^{\star} = \la_{l,n-k}, \quad \be_{l,k}^{\star} = \be_{l,n-k}, \quad l \ge 1, \, k = \overline{1,n-1}.
$$
\end{lem}

\begin{proof}
Relation \eqref{relM} implies 
\begin{equation} \label{relMk}
M^{\star}_{k+1,k}(\la) = M_{n-k+1,n-k}(\la), \quad k = \overline{1,n-1}.
\end{equation}

Let the assumptions (A-1) and (A-2) hold for $\{ \la_{l,k} \}_{l \ge 1, \, k = \overline{1,n-1}}$.
Taking \eqref{betaM} and $\be_{l,k} \ne 0$ into account, we conclude that all the poles of $M^{\star}_{k+1,k}(\la)$ are simple and coincide with $\{ \la_{l,n-k} \}_{l \ge 1}$. On the other hand, the poles of the $k$-th column of $M^{\star}(\la)$ belong to the set $\{ \la_{l,k}^{\star} \}_{l \ge 1}$. The set $\{ \la_{l,k}^{\star} \}_{l \ge 1} \setminus \{ \la_{l,n-k} \}_{l \ge 1}$ is empty because of the asymptotics \eqref{asymptla} for $\la_{l,k}^{\star}$ and $ \la_{l,n-k}$. Hence $\la_{l,k}^{\star} = \la_{l,n-k}$ for $l \ge 1$, $k = \overline{1,n-1}$. This implies (A-1) and (A-2) for $\{ \la_{l,k}^{\star} \}_{l \ge 1, \, k = \overline{1,n-1}}$. Thus $F^{\star} \in \mathfrak F_{n,simp}$. The relation $\be_{l,k}^{\star} = \be_{l,n-k}$ follows from \eqref{betaM} and \eqref{relMk}.
\end{proof}

\section{Self-adjoint case} \label{sec:sa}

Denote by $\mathfrak F_n^+$ and $\mathfrak F_{n,simp}^+$ the subclasses of matrix functions of $\mathfrak F_n$ and $\mathfrak F_{n,simp}$, respectively, satisfying the additional condition
\begin{equation} \label{condf}
f_{k,j}(x) = (-1)^{k+j+1} \overline{f_{n-j+1,n-k+1}(x)},
\end{equation}
where the bar denotes the complex conjugate.
The condition \eqref{condf} is a kind of self-adjointness. In particular, the matrix \eqref{F2} belongs to $\mathfrak F_2^+$ if $\sigma(x)$ is real-valued. In this section, we study the properties of the spectral data for $F \in \mathfrak F_{n,simp}^+$.

If $F \in \mathfrak F_n^+$, then $F^{\star}(x) = [\overline{f_{k,j}(x)}]_{k,j = 1}^n$. Consequently, 
$$
\mathcal C_k^{\star}(x, \la) = \overline{\mathcal C_k(x, (-1)^n\overline{\la})}, \quad \Phi_k^{\star}(x, \la) = \overline{\Phi_k(x, (-1)^n\overline{\la})}, \quad M_{j,k}^{\star}(\la) = \overline{M_{j,k}((-1)^n\overline{\la})}, \quad j,k = \overline{1,n}.
$$

In view of Lemma~\ref{lem:simps}, for $F \in \mathfrak F_{n,simp}^+$, we have
\begin{equation} \label{sasd}
\la_{l,k} = (-1)^n \overline{\la_{l,n-k}}, \quad
\be_{l,k} = (-1)^n \overline{\be_{l,n-k}}, \quad l \ge 1, \, k = \overline{1,n-1}.
\end{equation}

In particular, if $n = 2p$, $p \in \mathbb N$, then the boundary value problem $\mathcal L_p$ is self-adjoint and $\la_{l,p}$, $\be_{l,p}$ are real for all $l \ge 1$. Moreover, the following lemma holds.

\begin{lem} \label{lem:positive}
Suppose that $n = 2p$, $p \in \mathbb N$, and $F \in \mathfrak F_{n,simp}^+$. Then $(-1)^{p+1} \be_{l,p} > 0$ for all $l \ge 1$.
\end{lem}

\begin{proof} 
Fix $l \in \mathbb N$. By Lemma 3 in \cite{Bond22-alg}, we have the relation
$$
\Phi_{\langle -1 \rangle}(x, \la_{l,p}) = \Phi_{\langle 0 \rangle}(x, \la_{l,p}) \mathcal N(\la_{l,p}).
$$

In view of \eqref{structN} and \eqref{betaM}, this implies
\begin{equation} \label{smPhi}
\Phi_{p,\langle -1 \rangle}(x, \la_{l,p}) = \Phi_{p+1}(x, \la_{l,p}) \be_{l,p}.
\end{equation}

Note that $\Phi_{p+1}(x, \la_{l,p})$ is the eigenfunction of the problem $\mathcal L_p$ corresponding to the real eigenvalue $\la_{l,p}$. The identity \eqref{wron} implies
\begin{equation} \label{wronPhi}
\langle \Phi^{\star}_{p+1}(x, \la_{l,p}), \Phi_p(x, \la) \rangle \Big|_0^1 = (\la - \la_{l,p}) \int_0^1 \Phi_{p+1}^{\star}(x, \la_{l,p}) \Phi_p(x, \la) \, dx.
\end{equation}

Using \eqref{defLagr} and \eqref{bcPhi}, we calculate
$$
\langle \Phi^{\star}_{p+1}(x, \la_{l,p}), \Phi_p(x, \la) \rangle = \begin{cases}
    (-1)^p, & x = 0, \\
    0, & x = 1.
\end{cases}
$$

Consequently, it follows from \eqref{wronPhi} that
$$
(-1)^{p + 1} = \lim_{\la \to \la_{l,p}} (\la - \la_{l,p}) \int_0^1 \Phi_{p+1}^{\star}(x, \la_{l,p}) \Phi_p(x, \la) \, dx = 
\int_0^1 \Phi_{p+1}^{\star}(x, \la_{l,p}) \Phi_{p, \langle -1 \rangle}(x, \la_{l,p}) \, dx.
$$

Taking the relations \eqref{smPhi}, $\Phi^{\star}_{p+1}(x, \la_{l,p}) = \overline{\Phi_{p+1}(x, \overline{\la_{l,p}})}$, and $\la_{l,p} \in \mathbb R$ into account, we conclude that
$$
(-1)^{p + 1} = \be_{l,p} \int_0^1 |\Phi_{p+1}(x, \la_{l,p})|^2 \, dx.
$$

Since the integral is positive, this yields the claim.
\end{proof}

In addition, we obtain the following lemma for the eigenvalues corresponding to the matrix $F^0(x) = [\de_{k+1,j}]_{k,j = 1}^n$.

\begin{lem} \label{lem:0}
The eigenvalues $\{ \la_{l,k}^0 \}_{l \ge 1, \, k = \overline{1,n-1}}$ are real for all sufficiently large indices $l$.
\end{lem}

\begin{proof}
On the one hand, $(F^0)^{\star} = F^0$, so $\la_{l,k} = (-1)^n\la_{l,k}^{\star}$. Using Lemma~\ref{lem:simps}, we conclude that $\la_{l,k} = (-1)^n \la_{l,n-k}$. On the other hand, $F^0 \in \mathfrak F_n^+$, so $\la_{l,k} = (-1)^n \overline{\la_{l,n-k}}$. Consequently, the spectrum $\{ \la_{l,k} \}_{l \ge 1}$ of each problem $\mathcal L_k$ is symmetric with respect to the real axis. Furthermore, according to the asymptotics \eqref{asymptla}, the eigenvalues $\{ \la_{l,k} \}_{l \ge 1}$ are simple for all sufficiently large indices $l$, so they are real. 
\end{proof}

\section{Main equation} \label{sec:main}

In this section, we provide the main equation \eqref{main} of the method of spectral mappings for the system \eqref{sys} basing on the results of \cite{Bond22-alg}. First, we need some additional notations.

Consider two matrix functions $F(x)$ and $\tilde F(x)$ of the class $\mathfrak F_{n,simp}$. We agree that, if a symbol $\ga$ denotes an object related to $F(x)$, then the symbol $\tilde \ga$ with tilde will denote the analogous object related to $\tilde F(x)$. In addition, define the matrix $\tilde F^{\star}(x)$ similarly to \eqref{defFs}. Note that, for solutions related to the matrix functions $\tilde F(x)$ and $\tilde F^{\star}(x)$, the quasi-derivatives are defined similarly to \eqref{quasi} and \eqref{quasiz} with $\tilde f_{k,j}$ and $\tilde f_{k,j}^{\star}$ instead of $f_{k,j}$ and $f_{k,j}^{\star}$, respectively.
For technical simplicity, assume that 
\begin{equation} \label{difla}
\{ \la_{l,k} \}_{l \ge 1, \, k = \overline{1,n-1}} \cap 
\{ \tilde \la_{l,k} \}_{l \ge 1, \, k = \overline{1,n-1}} = \varnothing.
\end{equation}
The opposite case requires minor changes.

For convenience, introduce the notations
\begin{gather} \nonumber 
    V := \{ (l,k,\eps) \colon l \in \mathbb N, \, k = \overline{1,n-1}. \, \eps = 0, 1 \},  \\ \nonumber
    \la_{l,k,0} := \la_{l,k}, \quad \la_{l,k,1} := \tilde \la_{l,k}, \quad \be_{l,k,0} := \be_{l,k}, \quad \be_{l,k,1} := \tilde \be_{l,k}, \\ \label{defvv}
    \vv_{l,k,\eps}(x) := \Phi_{k+1}(x, \la_{l,k,\eps}), \quad \tilde \vv_{l,k,\eps}(x) := \tilde \Phi_{k+1}(x, \la_{l,k,\eps}), \quad (l,k,\eps) \in V.
\end{gather}

Thus, the indices $0$ and $1$ are used for the values related to $F(x)$ and $\tilde F(x)$, respectively. Note that the Weyl solution $\Phi_{k+1}(x, \la)$ and $\tilde \Phi_{k+1}(x, \la)$ have poles $\{ \la_{l,k+1,0} \}$ and $\{ \la_{l,k+1,1} \}$, respectively. Therefore, under the assumptions (A-2) and \eqref{difla}, the numbers $\{ \la_{l,k,\eps} \}$ are regular points of $\Phi_{k+1}(x, \la)$ and $\tilde \Phi_{k+1}(x, \la)$, so \eqref{defvv} correctly defines the functions $\vv_{l,k,\eps}(x)$ and $\tilde \vv_{l,k,\eps}(x)$.

Introduce the auxiliary functions
\begin{equation} \label{defD}
\tilde D_{k,k_0}(x, \mu, \la) := \frac{\langle\tilde \Phi_k^{\star}(x, \mu), \tilde \Phi_{k_0}(x, \la)\rangle}{\la - \mu}, \quad k, k_0 = \overline{1,n},
\end{equation}
where the Lagrange bracket is defined by \eqref{defLagr} and the quasi-derivatives for $\tilde \Phi_k^{\star}(x, \mu)$ and $\tilde \Phi_{k_0}(x, \la)$ are generated by the matrices $\tilde F^{\star}(x)$ and $\tilde F(x)$, respectively.

\begin{lem} \label{lem:singD}
The function $\tilde D_{k,k_0}(x, \mu, \la)$ has singularities at $\mu = \la_{l,n-k,1}$ if $k < n$, at $\la = \la_{l,k_0,1}$ if $k_0 < n$, and at $\la = \mu$ if $k + k_0 = n + 1$.
\end{lem}

\begin{proof}
Using \eqref{wron} and \eqref{defD}, we obtain
$$
\frac{d}{dx} \tilde D_{k_0,k}(x, \mu, \la) = \tilde \Phi_k^{\star}(x, \mu) \tilde \Phi_{k_0}(x, \la).
$$
The boundary conditions \eqref{bcPhi} for $\tilde \Phi_k^{\star}(x, \mu)$ and $\tilde \Phi_{k_0}(x, \la)$ together with \eqref{defLagr} imply
\begin{align*}
\langle \tilde \Phi_k^{\star}(x, \mu), \tilde \Phi_{k_0}(x, \la) \rangle_{|x = 0} & = \begin{cases}
                        0, &  k + k_0 > n + 1, \\
                        (-1)^{k+1}, & k + k_0 = n + 1,
                    \end{cases}, \\
\langle \tilde \Phi_k^{\star}(x, \mu), \tilde \Phi_{k_0}(x, \la) \rangle_{|x = 1} & = 0, \quad k + k_0 < n + 1.
\end{align*}
Hence
$$
\tilde D_{k_0,k}(x, \mu, \la) = \begin{cases}
\int_0^x \tilde \Phi_k^{\star}(t, \mu) \tilde \Phi_{k_0}(t, \la) \, dt, & k + k_0 > n + 1, \\
\frac{(-1)^{k+1}}{\la - \mu} + \int_0^x \tilde \Phi_k^{\star}(t, \mu) \tilde \Phi_{k_0}(t, \la) \, dt, & k + k_0 = n + 1, \\
-\int_x^1 \tilde \Phi_k^{\star}(t, \mu) \tilde \Phi_{k_0}(t, \la) \, dt, & k + k_0 < n + 1.
\end{cases}
$$
Consequently, the singularities of $\tilde D_{k_0,k}(x, \la)$ coincide with the poles $\{ \tilde \la_{l,k}^{\star} \}_{l \ge 1}$ of $\tilde \Phi_k^{\star}(t, \mu)$ for $k < n$ and with the poles $\{ \tilde \la_{l,k_0} \}_{l \ge 1}$ of $\tilde \Phi_{k_0}(t,\la)$ for $k_0 < n$. In addition, $\la = \mu$ is a pole in the case $k + k_0 = n + 1$. By Lemma~\ref{lem:simps}, $\tilde \la_{l,k}^{\star} = \tilde \la_{l,n-k}$, which concludes the proof.
\end{proof}

For $(l,k,\eps), (l_0,k_0,\eps_0) \in V$, denote
\begin{equation}
\label{defG}
    \tilde G_{(l,k,\eps), (l_0, k_0, \eps_0)}(x) :=  
    (-1)^{n-k} \be_{l,k,\eps} \tilde D_{n-k+1, k_0 + 1}(x, \la_{l,k,\eps}, \la_{l_0,k_0,\eps_0}).
\end{equation}

By Lemma~\ref{lem:singD}, $(\mu, \la) = (\la_{l,k,\eps}, \la_{l_0,k_0,\eps_0})$ is a regular point of $\tilde D_{n-k+1,k_0 + 1}(x, \mu, \la)$, so the definition \eqref{defG} is correct.

\begin{prop}[\cite{Bond22-alg}]
The following relations hold:
\begin{align} \label{infphi}
\vv_{l_0, k_0,\eps_0}(x) & = \tilde \vv_{l_0,k_0,\eps_0}(x) + \sum_{(l,k,\eps) \in V}(-1)^{\eps} \vv_{l,k,\eps}(x) \tilde G_{(l,k,\eps), (l_0, k_0,\eps_0)}(x), \quad (l_0,k_0,\eps_0) \in V, \\ \label{recPhik}
\Phi_{k_0}(x, \la) & = \tilde \Phi_{k_0}(x, \la) + \sum_{(l,k,\eps) \in V} (-1)^{\eps + n - k} \be_{l,k,\eps} \tilde \vv_{l,k,\eps}(x) \tilde D_{n-k+1,k_0}(x, \la_{l,k,\eps}, \la), \quad k_0 = \overline{1,n}.
\end{align}
\end{prop}

The relation \eqref{infphi} can be treated as an infinite system of linear equations with respect to $\{ \vv_{l,k,\eps}(x) \}_{(l,k,\eps) \in V}$ for each fixed $x \in [0,1]$. This system plays an important role for solving inverse spectral problems (see \cite{Bond22-alg, Bond23-loc}). The relation \eqref{recPhik} can be used for finding the Weyl solutions $\{ \Phi_k(x, \la)\}_{k = 1}^n$ from the solution $\{ \vv_{l,k,\eps}(x) \}_{(l,k,\eps) \in V}$ of the system \eqref{infphi}. In order to study the solvability of the system \eqref{infphi}, we reduce it to a linear equation in a suitable Banach space by the method of \cite{Yur02, Bond22-alg}.

Define the numbers
\begin{equation} \label{defxi}
\xi_l  := \sum_{k = 1}^{n-1} \left( l^{-(n-1)} |\la_{l,k} - \tilde \la_{l,k}| + l^{-n} |\be_{l,k} - \tilde \be_{l,k}| \right), \quad l \ge 1, 
\end{equation}
and the functions
\begin{equation} \label{defw}
w_{l,k}(x) := l^{-k} \exp(-xl \cot(k\pi/n)).
\end{equation}
The numbers $\{ \xi_l \}_{l \ge 1}$ characterize the difference between the spectral data $\{ \la_{l,k}, \be_{l,k} \}_{l \ge 1, \, k = \overline{1,n-1}}$ and $\{ \tilde \la_{l,k}, \tilde \be_{l,k} \}_{l \ge 1, \, k = \overline{1,n-1}}$. The functions $w_{l,k}(x)$ are related to the growth of the functions $\vv_{l,k,\eps}(x)$: $|\vv_{l,k,\eps}(x)| \le C w_{l,k}(x)$. The latter estimate can be easily deduced from Proposition~\ref{prop:estPhi} and the asymptotics \eqref{asymptla}.

Apply the following transform to the functions in the system \eqref{infphi}:
\begin{equation} \label{defpsi}
\begin{bmatrix}
\psi_{l,k,0}(x) \\ \psi_{l,k,1}(x)
\end{bmatrix} := 
w_{l,k}^{-1}(x)
\begin{bmatrix}
\xi_l^{-1} & -\xi_l^{-1} \\ 0 & 1
\end{bmatrix}
\begin{bmatrix}
\vv_{l,k,0}(x) \\ \vv_{l,k,1}(x)
\end{bmatrix}, 
\end{equation}
\begin{multline} \label{defR}
\begin{bmatrix}
\tilde R_{(l_0,k_0,0),(l,k,0)}(x) & \tilde R_{(l_0,k_0,0),(l,k,1)}(x) \\
\tilde R_{(l_0,k_0,1),(l,k,0)}(x) & \tilde R_{(l_0,k_0,1),(l,k,1)}(x)
\end{bmatrix} := \\ 
\frac{w_{l,k}(x)}{w_{l_0,k_0}(x)}
\begin{bmatrix}
\xi_{l_0}^{-1} & -\xi_{l_0}^{-1} \\ 0 & 1
\end{bmatrix}
\begin{bmatrix}
\tilde G_{(l,k,0),(l_0,k_0,0)}(x) & \tilde G_{(l,k,1),(l_0,k_0,0)}(x) \\
\tilde G_{(l,k,0),(l_0,k_0,1)}(x) & \tilde G_{(l,k,1),(l_0,k_0,1)}(x)
\end{bmatrix}
\begin{bmatrix}
\xi_l & 1 \\ 0 & -1
\end{bmatrix}.
\end{multline}
Analogously to $\psi_{l,k,\eps}(x)$, define $\tilde \psi_{l,k,\eps}(x)$. 

For brevity, denote $v = (l,k,\eps)$, $v_0 = (l_0,k_0,\eps_0)$, $v,v_0 \in V$. Define the vectors $\psi(x) := [\psi_v(x)]_{v \in V}$ and $\tilde \psi(x) = [\tilde \psi_v(x)]_{v \in V}$.

Consider the Banach space $m$ of bounded infinite sequences $\al = [\al_v]_{v \in V}$ with the norm $\| \al \|_m = \sup_{v \in V}|\al_v|$. Define the linear operator $\tilde R(x) = [\tilde R_{v_0,v}(x)]_{v_0, v \in V}$ acting on an element $\al = [\al_v]_{v \in V}$ of $m$ by the following rule:
$$
(\tilde R(x) \al)_{v_0} = \sum_{v \in V} \tilde R_{v_0,v}(x) \al_v.
$$

\begin{prop}[\cite{Bond22-alg}] \label{prop:Rpsi}
For each fixed $x \in [0,1]$, the vectors $\psi(x)$ and $\tilde \psi(x)$ belong to $m$ and $\tilde R(x)$ is a bounded operator from $m$ to $m$. Moreover, the operator $\tilde R(x)$ can be approximated by finite-dimensional operators with respect to the operator norm $\| . \|_{m \to m}$, so $\tilde R(x)$ is compact.
\end{prop}

\begin{prop}[\cite{Bond22-alg}] \label{prop:maineq}
Suppose that $F, \tilde F \in \mathfrak F_{n,simp}$. Let $\psi(x)$, $\tilde \psi(x)$, and $\tilde R(x)$ be constructed by using the matrix functions $F(x)$, $\tilde F(x)$ and their spectral data as described above. Then, for each fixed $x \in [0,1]$, the following relation is fulfilled
\begin{equation} \label{main}
(I - \tilde R(x)) \psi(x) = \tilde \psi(x)
\end{equation}
in the Banach space $m$, where $I$ is the identity operator. Furthermore, for each fixed $x \in [0,1]$, the operator $(I - \tilde R(x))$ has a bounded inverse, so equation \eqref{main} is uniquely solvable with respect to $\psi(x)$.
\end{prop}

The relation \eqref{main} is called \textit{the main equation} of the inverse problem. Obviously, \eqref{main} is deduced from the system \eqref{infphi} by using the notations \eqref{defpsi} and \eqref{defR}. It is worth mentioning that, in \cite{Bond22-alg}, the results of Propositions \ref{prop:Rpsi} and \ref{prop:maineq} were obtained for the general case without the separation assumption (A-2). However, this assumption simplifies the form of the functions $\tilde G_{(l,k,\eps),(l_0,k_0,\eps_0)}(x)$ and $\tilde R_{v_0,v}(x)$, which will be used in the proofs. Furthermore, it is important to note that the unique solvability of the main equation \eqref{main} was proved in \cite{Bond22-alg} under the assumption that $\{ \la_{l,k}, \be_{l,k} \}$ are the spectral data of some problems $\mathcal L_k$, $k = \overline{1,n-1}$. In this case, the inverse operator $(I - \tilde R(x))^{-1}$ can be found explicitly (see \cite[Theorem~1]{Bond22-alg}). But, in this paper, we will consider the main equation \eqref{main} constructed by numbers $\{ \la_{l,k}, \be_{l,k} \}$ that are not necessarily related to some matrix function $F(x)$ and obtain sufficient conditions for the invertibility of the operator $(I - \tilde R(x))$.

\section{Main equation solvability} \label{sec:inverse}

In this section, we prove the unique solvability of the main equation \eqref{main} under some simple conditions on the given data $\{ \la_{l,k}, \be_{l,k} \}_{l \ge 1\, k= \overline{1,n-1}}$. These conditions include only the asymptotics and some structural properties. We emphasize that the numbers $\{ \la_{l,k}, \be_{l,k} \}_{l \ge 1\, k= \overline{1,n-1}}$ are not assumed to be the spectral data of the system \eqref{sys}. The main results are formulated in Theorem~\ref{thm:inverse}. Its proof contains several lemmas and their proofs inside of it. A reader can skip the proofs of those auxiliary lemmas to get the main idea.
The central role in the proofs is played by the meromorphic functions $B_j(x, \la)$, $j = \overline{1,n}$ defined by \eqref{defB}. On the one hand, these functions are estimated as $|\la| \to \infty$ and it is shown that their integrals over large contours tend to zero. On the other hand, those integrals are calculated by using the Residue Theorem. This idea arises from the proofs for $n = 2$ (see Lemma 1.3.6 in \cite{Yur02} and Lemma 5.2 in \cite{Bond21-tamkang}) and $n = 3$ (see Lemma 6.1 in \cite{Bond23-res}). However, the generalization to the case of arbitrary integer $n$ requires much technical work.

\begin{thm} \label{thm:inverse}
Suppose that $\tilde F \in \mathfrak F_{n,simp}^+$ and complex numbers $\{ \la_{l,k}, \be_{l,k} \}_{l \ge 1, \, k = \overline{1,n}}$ satisfy the assumptions (A-1) and (A-2), the asymptotics \eqref{asymptla} and \eqref{asymptbe}, the self-adjointness conditions \eqref{sasd}, the additional requirements
\begin{equation} \label{addhyp}
\begin{array}{l}
\text{if} \:\: n = 2p \colon  \quad (-1)^{p+1}\be_{l,p} > 0, \quad l \ge 1, \\
\text{if} \:\: n = 2p+1 \colon  \quad (-1)^{p+1} \mbox{Re}\, \la_{l,p} > 0,  \quad l \ge 1,
\end{array}
\end{equation}
and $\be_{l,k} \ne 0$ for $l \ge 1$, $k = \overline{1,n-1}$. Then, the linear operator $(I - \tilde R(x))$, which is constructed according to Section~\ref{sec:main}, has a bounded inverse operator in the Banach space $m$ for each fixed $x \in [0,1]$.
\end{thm}

\begin{proof}[Proof of Theorem~\ref{thm:inverse}]
Fix $x \in [0,1]$. Consider the operator $(I - \tilde R(x))$ satisfying the hypotheses of the theorem. By virtue of Proposition~\ref{prop:Rpsi}, the operator $\tilde R(x)$ possesses the approximation property, so the Fredholm Theorem can be applied. Therefore, it is sufficient to prove that the homogeneous equation
\begin{equation} \label{homo}
(I - \tilde R(x)) \zeta(x) = 0, 
\end{equation}
has the only solution $\zeta(x) = 0$ in $m$.

Let $\zeta(x) = [\zeta_v(x)]_{v \in V} \in m$ be a solution of \eqref{homo}. This means
\begin{equation} \label{syszeta}
\zeta_{v_0}(x) = \sum_{v \in V} \tilde R_{v_0, v}(x) \zeta_v(x), \quad v_0 \in V.
\end{equation}

Apply the transform
$$
\begin{bmatrix}
z_{l,k,0}(x) \\ z_{l,k,1}(x)
\end{bmatrix} := w_{l,k}(x) 
\begin{bmatrix}
\xi_l & 1 \\
0 & 1
\end{bmatrix}
\begin{bmatrix}
\zeta_{l,k,0}(x) \\
\zeta_{l,k,1}(x)
\end{bmatrix},
$$
which is inverse to the transform in \eqref{defpsi}. Using \eqref{syszeta} and \eqref{defR}, we obtain the infinite system
\begin{equation} \label{infz}
z_{l_0,k_0,\eps_0}(x) = \sum_{(l,k,\eps) \in V} (-1)^{\eps} z_{l,k,\eps}(x) \tilde G_{(l,k,\eps), (l_0,k_0,\eps_0)}(x), \quad (l_0, k_0, \eps_0) \in V,
\end{equation}
which is the homogeneous analog of \eqref{infphi}. Since $\zeta(x) \in m$, then $|\zeta_{l,k,\eps}(x)| \le C$ and so
\begin{equation} \label{estz}
|z_{l,k,\eps}(x)| \le C w_{l,k}(x), \quad |z_{l,k,0}(x) - z_{l,k,1}(x)| \le C \xi_l w_{l,k}(x), \quad (l,k,\eps) \in V,
\end{equation}
where $\xi_l$ and $w_{l,k}(x)$ are defined in \eqref{defxi} and \eqref{defw}, respectively.

Introduce the functions
\begin{equation} \label{defZ}
Z_{k_0}(x, \la) := \sum_{(l,k,\eps) \in V} (-1)^{\eps + n -k} \be_{l,k,\eps} z_{l,k,\eps}(x) \tilde D_{n-k+1,k_0}(x, \la_{l,k,\eps}, \la), \quad k_0 = \overline{1,n},
\end{equation}
analogously to \eqref{recPhik}. It follows from \eqref{infz} and \eqref{defZ} that 
\begin{equation} \label{relZzk}
Z_{k + 1}(x, \la_{l,k,\eps}) = z_{l,k,\eps}(x), \quad (l,k,\eps) \in V.
\end{equation}

The following lemma shows that the functions $\{ Z_k(x, \la) \}_{k = 1}^n$ have the same growth for $|\la| \to \infty$ as the corresponding Weyl solutions $\{ \Phi_k(x, \la) \}_{k = 1}^n$ (see Proposition~\ref{prop:estPhi}).

\begin{lem} \label{lem:estZ}
For $k_0 = \overline{1,n}$, the function $Z_{k_0}(x, \la)$ satisfies the estimate
\begin{align} \label{estZ}
|Z_{k_0}(x, \rho^n)| & \le \sum_{l = 1}^{\infty} \sum_{k = 1}^{n-1} \frac{C \xi_l |\rho|^{-(k_0 - 1)} |\exp(\rho \om_{k_0} x)|}{|\rho - c_k l| + 1}, \\ \label{estZ2} & \le C |\rho|^{-(k_0-1)} |\exp(\rho \om_{k_0} x)|,
\quad x \in [0,1], \quad 
\rho \in \overline{\Gamma}_{s,\rho^*,\de},
\end{align}
for each fixed $s = \overline{1,m}$, a sufficiently small $\de > 0$, and some $\rho^* > 0$, where the region $\Gamma_{s, \rho^*, \de}$ was defined in \eqref{defGs}, the roots $\{ \om_k \}_{k = 1}^n$ are numbered in the order \eqref{order} associated with the sector $\Gamma_s$, and $\{ c_k \}_{k = 1}^{n-1}$ are such constants that $\rho_{l,k}^0 \sim c_k l$ as $l \to \infty$. (In view of the asymptotics \eqref{asymptla}, $c_k = \frac{\pi}{\sin\frac{\pi k}{n}} \epsilon_k$, where $|\epsilon_k| = 1$ and $\arg \epsilon_k$ depends on $\Gamma_s$).
\end{lem}

\begin{proof}
Let us estimate the series in \eqref{defZ}, which can be represented as follows:
\begin{align} \nonumber
Z_{k_0}(x) = & \sum_{l = 1}^{\infty} \sum_{k = 1}^{n-1} (-1)^{n-k} \bigl( (\be_{l,k,0} - \be_{l,k,1}) z_{l,k,0}(x) \tilde D_{n-k+1, k_0}(x, \la_{l,k,0}, \la) \\ \nonumber
& + \be_{l,k,1} (z_{l,k,0}(x) - z_{l,k,1}(x)) \tilde D_{n-k+1, k_0}(x, \la_{l,k,0}, \la) \\ \label{seriesZ}
& + \be_{l,k,1} z_{l,k,1}(x) (\tilde D_{n-k+1, k_0}(x, \la_{l,k,0}, \la) - \tilde D_{n-k+1,k_0}(x, \la_{l,k,1}, \la) \bigr).
\end{align}

We begin with the functions $\tilde D_{n-k+1, k_0}(x, \la_{l,k,\eps}, \la)$. Due to \eqref{defD} and \eqref{defLagr}, we have
\begin{align} \nonumber
\tilde D_{n-k+1,k_0}(x, \la_{l,k,\eps}, \la) & = \frac{\langle \tilde \Phi^{\star}_{n-k+1}(x, \la_{l,k,\eps}), \tilde \Phi_{k_0}(x, \la)\rangle}{\la - \la_{l,k,\eps}} \\ \label{relD}
& = \frac{1}{\la - \la_{l,k,\eps}} \sum_{j = 0}^{n-1} (-1)^j \tilde \Phi_{n-k+1}^{\star[j]}(x, \la_{l,k,\eps}) \tilde \Phi_{k_0}^{[n - j - 1]}(x, \la).
\end{align}

Obviously, the estimate of Proposition~\ref{prop:estPhi} is valid for $\tilde \Phi^{\star}_{n-k+1}(x, \la)$. Taking the asymptotics \eqref{asymptla} for $\la_{l,k,\eps}$ and the definition \eqref{defw} into account, we obtain
\begin{equation} \label{estPhi*}
|\tilde \Phi_{n-k+1}^{\star[j]}(x, \la_{l,k,\eps})| \le C |\rho_{l,k,\eps}|^{j - (n-k)} |\exp(\rho_{l,k,\eps} \om_{n-k+1} x)| \le C l^{j - n} w_{l,k}^{-1}(x),
\end{equation}
where $\rho_{l,k,\eps} = \sqrt[n]{\la_{l,k,\eps}}$.

Using \eqref{asymptla}, \eqref{relD}, \eqref{estPhi*}, and Proposition~\ref{prop:estPhi}, we get
\begin{align} \nonumber
|\tilde D_{n-k+1,k_0}(x, \la_{l,k,\eps}, \la)| & \le \frac{C \sum\limits_{j = 0}^{n-1} l^{j-n} w_{l,k}^{-1}(x) |\rho|^{n-j-k_0}|\exp(\rho \om_{k_0} x)|}{|\rho - \rho_{l,k,\eps}| \sum\limits_{j = 0}^{n-1}|\rho|^{n-j-1} l^j} \\ \label{estD1} & \le \frac{C |\rho|^{-(k_0-1)} l^{-n} w_{l,k}^{-1}(x) |\exp(\rho \om_{k_0} x)|}{|\rho - \rho_{l,k}^0|}, \quad \rho \in \overline{\Gamma}_{s, \rho^*, \de}, \quad \la = \rho^n.
\end{align}

It follows from \eqref{defxi} that $|\rho_{l,k,0} - \rho_{l,k,1}| \le C \xi_l$, where we choose the same branch of the root $\sqrt[n]{\la_{l,k,\eps}}$ for $\eps = 0, 1$.
Consequently, using the standard approach based on Schwarz's Lemma (see \cite[Lemmas 1.3.1 and 1.3.2]{Yur02}), we estimate the difference
\begin{equation} \label{estD2}
|\tilde D_{n-k+1, k_0}(x, \la_{l,k,0}, \la) - \tilde D_{n-k+1, k_0}(x, \la_{l,k,1}, \la)| \le \frac{C |\rho|^{-(k_0-1)} \xi_l l^{-n} w_{l,k}^{-1}(x) |\exp(\rho \om_{k_0} x)|}{|\rho - \rho_{l,k}^0|}
\end{equation}
for $\rho \in \overline{\Gamma}_{s, \rho^*, \de}$. In view of \eqref{asymptbe} and \eqref{defxi}, we have
\begin{equation} \label{estbeta}
|\be_{l,k,\eps}| \le C l^n, \quad |\be_{l,k,0} - \be_{l,k,1}| \le C l^n \xi_l.
\end{equation}

Using the estimates \eqref{estz}, \eqref{estD1}, \eqref{estD2}, and \eqref{estbeta} together with \eqref{seriesZ}, we arrive at \eqref{estZ}.

According to the asymptotics \eqref{asymptla} and \eqref{asymptbe} for the data $\{ \la_{l,k}, \be_{l,k} \}_{l \ge 1, \, k = \overline{1,n-1}}$ and
$\{ \tilde \la_{l,k}, \tilde \be_{l,k} \}_{l \ge 1, \, k = \overline{1,n-1}}$, we have $\{ \xi_l \} \in l_2$. Hence, the Cauchy-Bunyakovsky-Schwartz inequality implies
$$
\sum_{l,k} \frac{\xi_l}{|\rho - c_k l| + 1} \le \sqrt{\sum_l \xi_l^2} \sqrt{\sum_{l,k} \frac{1}{(|\rho - c_k l| + 1)^2}} \le C,
$$
which proves the estimate \eqref{estZ2}.
\end{proof}

Our next goal is to study analytic properties of the functions $Z_k(x, \la)$, $k = \overline{1,n}$. For this purpose, we consider the auxiliary functions
\begin{align} \label{defE}
& \tilde E_{k,k_0}(x, \mu, \la) := \frac{\langle \tilde \Phi^{\star}_k(x, \mu), \tilde {\mathcal C}_{k_0}(x, \la) \rangle}{\la - \mu}, \\ \label{defz}
& z_{k_0}(x, \la) := \sum_{(l,k,\eps) \in V} (-1)^{\eps + n - k} \be_{l,k,\eps} z_{l,k,\eps}(x) \tilde E_{n-k+1, k_0}(x, \la_{l,k,\eps}, \la).
\end{align}

Thus, the functions $z_{k_0}(x, \la)$ are defined analogously to $Z_{k_0}(x, \la)$ by replacing $\tilde \Phi_{k_0}(x, \la)$ by $\tilde{\mathcal C}_{k_0}(x, \la)$. It is easier to consider the functions $z_{k_0}(x, \la)$ than $Z_{k_0}(x, \la)$, because the functions $\{ \tilde{\mathcal C_k}(x, \la) \}_{k = 1}^n$ are entire in $\la$. Without loss of generality we assume that $\la_{l,k,\eps} \ne \la_{l_0,k_0,\eps}$ for $l \ne l_0$.

\begin{lem} \label{lem:resz}
For $k_0 = \overline{1,n}$ and each fixed $x \in [0,1]$, the function $z_{k_0}(x, \la)$ is analytic in the $\la$-plane except for the simple poles $\{ \la_{l,k,\eps} \}_{l \ge 1,  \, k = \overline{k_0,n-1}}$. In particular, $z_n(x,\la)$ is entire in $\la$. Moreover,
\begin{equation} \label{resz}
\Res_{\la = \la_{l,k,\eps}} z_{k_0}(x, \la) = \sum_{i = 1}^s (-1)^{\eps + k_0 - k_i} \be_{l,k_i,\eps} z_{l,k_i,\eps}(x) \tilde M^{\star}_{n - k_0 + 1, n - k_i + 1}(\la_{l,k_i,\eps}), \quad (l,k,\eps) \in V, \: k \ge k_0,
\end{equation}
where, for a fixed triple $(l,k,\eps) \in V$ and a fixed $k_0$, $\{ k_i \}_{i = 1}^s$ is the set of all the indices such that $\la_{l,k,\eps} = \la_{l,k_i,\eps}$ and $k_i \ge k_0$.
\end{lem}

\begin{proof}
Using \eqref{wron} and \eqref{defE}, we get
$$
\tilde E_{n-k+1, k_0}(x, \la_{l,k,\eps}, \la) = \frac{\langle \tilde \Phi^{\star}_{n-k+1}(x, \la_{l,k,\eps}), \tilde{\mathcal C}_{k_0}(x, \la) \rangle_{|x = 0}}{\la - \la_{l,k,\eps}} + \int_0^x \tilde \Phi^{\star}_{n-k+1}(t, \la_{l,k,\eps}) \tilde{\mathcal C}_{k_0}(t, \la) \, dt,
$$
where the integral, obviously, is entire in $\la$. Using \eqref{defLagr} and \eqref{initC}, we deduce
\begin{align*}
\langle \tilde \Phi^{\star}_{n-k+1}(x, \la_{l,k,\eps}), \tilde{\mathcal C}_{k_0}(x, \la) \rangle_{|x = 0} & = \sum_{j = 0}^{n-1} (-1)^j \tilde \Phi_{n-k+1}^{\star[j]}(0, \la_{l,k,\eps}) \tilde{\mathcal C}_{k_0}^{[n-j-1]}(0, \la) \\ & = (-1)^{n-k_0} \tilde \Phi_{n-k+1}^{\star[n-k_0]}(0, \la_{l,k,\eps}).
\end{align*}

The relation \eqref{defM*} in the element-wise form implies
$$
\tilde \Phi_r^{\star}(x, \la) = \tilde C_r^{\star}(x,\la) + \sum_{j = r+1}^n \tilde M_{j,r}^{\star}(\la)\tilde{\mathcal C}_j^{\star}(x, \la), \quad r = \overline{1,n}.
$$

Taking the initial conditions \eqref{initC} for $\tilde C_r^{\star}(x, \la)$ into account, we conclude that 
$$
\tilde \Phi_{n-k+1}^{\star[n - k_0]}(0, \la_{l,k,\eps}) = \tilde M_{n-k_0 + 1, n-k+1}^{\star}(\la_{l,k,\eps}).
$$
Note that this value equals zero for $k < k_0$. Consequently, the function $\tilde E_{n-k+1,k_0}(x, \la_{l,k,\eps}, \la)$ is analytic in $\la$ except for the simple pole $\la_{l,k,\eps}$ if $k \ge k_0$ and
\begin{equation} \label{resE}
\Res_{\la = \la_{l,k,\eps}} \tilde E_{n-k+1, k_0}(x, \la_{l,k,\eps}, \la) = (-1)^{n-k_0} \tilde M^{\star}_{n-k_0 + 1, n - k + 1}(\la_{l,k,\eps}).
\end{equation}

Combining \eqref{defz} and \eqref{resE}, we arrive at \eqref{resz}.
\end{proof}

Let us apply Lemma~\ref{lem:resz} to study analytic properties of $Z_k(x, \la)$.

\begin{lem} \label{lem:resZ}
For each fixed $k \in \{ 1, 2, \dots, n-1 \}$ and $x \in [0,1]$, the function $Z_k(x, \la)$ is analytic in $\la$ except for the simple poles $\{ \la_{l,k,0} \}_{l \ge 1}$. Moreover,
\begin{equation} \label{resZ}
\Res_{\la = \la_{l,k,0}} Z_k(x, \la) = \be_{l,k,0} z_{l,k,0}(x), \quad l \ge 1.
\end{equation}
The function $Z_n(x, \la)$ is entire in $\la$.
\end{lem}

\begin{proof}
It follows from Lemma~\ref{lem:estZ} that the series \eqref{defZ} converges absolutely and uniformly for $\rho$ on compact sets in $\overline{\Gamma}_{s,\rho^*, \de}$ and $\la = \rho^n$. Consequently, the functions $Z_{k_0}(x, \la)$, $k_0 = \overline{1,n}$, are analytic for such values of $\la$. Moreover, these functions can be analytically continued inside the circles that are cut out in $\overline{\Gamma}_{s,\rho^*, \de}$ with the possible exception of the values $\{ \la_{l,k,\eps} \}$. Therefore, it remains to compute the residues of $Z_{k_0}(x, \la)$ at these points.

Using the relation \eqref{defM} and Proposition~\ref{prop:M}, we obtain
$\tilde {\mathcal C}(x, \la) = \tilde \Phi(x, \la) (\tilde M(\la))^{-1}$ and so
\begin{equation} \label{relCPhi}
\tilde{\mathcal C}_{k_0}(x, \la) = \tilde \Phi_{k_0}(x, \la) + \sum_{j = k_0+1}^n (-1)^{j-k_0} \tilde M^{\star}_{n-k_0+1, n-j+1}(\la) \tilde \Phi_j(x, \la), \quad k_0 = \overline{1,n}.
\end{equation}

Substituting \eqref{relCPhi} into \eqref{defZ}, we derive
\begin{equation} \label{relZz}
Z_{k_0}(x, \la) = z_{k_0}(x, \la) - \sum_{j = {k_0}+1}^n (-1)^{j-{k_0}} \tilde M^{\star}_{n-{k_0}+1,n-j+1} (\la) Z_j(x, \la), \quad k_0 = \overline{1,n}.
\end{equation}

Let us prove the assertion of the lemma by induction for $k_0 = n, n-1, \dots, 2, 1$. For $k_0 = n$, the function $Z_n(x, \la) \equiv z_n(x, \la)$ is entire in $\la$ by virtue of Lemma~\ref{lem:resz}. Next, suppose that the assertion is already proved for $Z_{k_0+1}(x, \la)$, \dots, $Z_n(x, \la)$. Let us prove it for $Z_{k_0}(x, \la)$. Fix $(l,k,\eps) \in V$ and denote by $\{ k_i \}_{i = 1}^s$ the set of all the indices such that $\la_{l,k,\eps} = \la_{l,k_i,\eps}$, $k_i \ge k_0$, as in the statement of Lemma~\ref{lem:resz}. Consider the two cases.

\smallskip

\textit{Case 1:} $\eps = 0$. In view of \eqref{difla}, $\la_{l,k,0}$ is a regular point of $\tilde M^{\star}_{n-k_0+1,n-j+1}(\la)$. Using \eqref{resz}, \eqref{relZz}, and the induction hypothesis, we obtain
\begin{align*}
Z_{k_0, \langle -1 \rangle}(x, \la_{l,k,0}) = & \sum_{i = 1}^s (-1)^{ k_0 - k_i} \be_{l,k_i,0} z_{l,k_i,0}(x) \tilde M^{\star}_{n - k_0 + 1, n - k_i + 1}(\la_{l,k_i,0}) \\ & - \sum_{\substack{i = 1 \\ k_i \ne k_0}}^s (-1)^{k_i - k_0} \tilde M^{\star}_{n-k_0+1,n-k_i+1}(\la_{l,k_i,0}) Z_{k_i, \langle -1 \rangle}(x, \la_{l,k_i,0}).
\end{align*}

Thus, we obtain the formula \eqref{resZ} for $\la_{l,k,0} = \la_{l,k_0,0}$ and $Z_{k_0,\langle -1 \rangle}(x, \la_{l,k,0}) = 0$ otherwise.

\smallskip

\textit{Case 2:} $\eps = 1$. By the induction hypothesis, the functions $\{ Z_j(x, \la) \}_{j = k_0 + 1}^n$ are analytic at $\la_{l,k,1}$. The functions $\tilde M^{\star}_{n-k_0+1,n-j+1}(\la)$ have a pole $\la_{l,k,1}$ if $j = k_i + 1$, $i = \overline{1,s}$. Therefore, using \eqref{relZz} and \eqref{resz}, we obtain
\begin{align} \nonumber
Z_{k_0,\langle -1 \rangle}(x, \la_{l,k,1}) = & \sum_{i = 1}^s (-1)^{ k_0 - k_i + 1} \be_{l,k_i,1} z_{l,k_i,1}(x) \tilde M^{\star}_{n - k_0 + 1, n - k_i + 1}(\la_{l,k_i,1}) \\  \label{smZk} & - \sum_{i = 1}^s (-1)^{k_i - k_0 + 1} \tilde M^{\star}_{n-k_0 + 1, n - k_i, \langle -1 \rangle}(\la_{l,k_i,1}) Z_{k_i + 1}(x, \la_{l,k_i,1}).
\end{align}
By Lemma~\ref{lem:simps} 
\begin{equation} \label{tsimps}
\la_{l,k_i,1} = \tilde \la_{l,n-k_i}^{\star}, \quad \be_{l,k_i,1} = \tilde \be_{l,n-k_i}^{\star}.
\end{equation}
Substituting \eqref{relZzk} and \eqref{tsimps} into \eqref{smZk}, we arrive at the relation
\begin{align} \nonumber
Z_{k_0,\langle -1 \rangle}(x,\la_{l,k,1}) = & \sum_{i = 1}^s (-1)^{k_0 - k_i + 1} z_{l,k_i,1}(x) \\ \label{sm2} & \times \bigl( \tilde \be_{l,n-k_i}^{\star} \tilde M^{\star}_{n-k_i + 1, n-k_i, \langle 0 \rangle}(\tilde \la^{\star}_{l,n-k_i}) - \tilde M^{\star}_{n-k_0 + 1, n-k_i}(\tilde \la_{l,n-k_i}^{\star})\bigr).
\end{align}
It follows from \eqref{defN} that
\begin{equation} \label{relNM}
\tilde M_{\langle 0 \rangle}^{\star}(\tilde \la_{l,n-k_i}^{\star}) \tilde {\mathcal N}^{\star}(\tilde \la_{l,n-k_i}^{\star}) = \tilde M_{\langle -1 \rangle}^{\star}(\tilde \la_{l,n-k}^{\star}).
\end{equation}
By virtue of Lemma~\ref{lem:simps}, $\tilde F^{\star} \in \mathfrak F_{n, simp}$, so the matrices $\tilde {\mathcal N}^{\star}(\tilde \la_{l,n-k_i}^{\star})$ have the special structure \eqref{structN}. Therefore, the relation \eqref{relNM} implies that the expression in the brackets in \eqref{sm2} vanishes.
Hence, $\la_{l,k,1}$ is a regular point of $Z_{k_0}(x, \la)$. By induction, this concludes the proof.
\end{proof}

We need to show that $z_{l,k,\eps}(x) = 0$ for all $(l,k,\eps) \in V$. Fir this purpose, we will use the following two lemmas.

\begin{lem} \label{lem:Zup}
If $z_{l,n-p,0}(x) = 0$ for all $l \ge 1$, then $z_{l,k,\eps}(x) = 0$ for $(l,k,\eps) \in V$, $k = \overline{n-p, n-1}$.
\end{lem}

\begin{proof}
Suppose that $z_{l,k,0}(x) = 0$ for some $k \ge n-p$ and all $l \ge 1$. In view of \eqref{relZz} and Lemma~\ref{lem:resZ}, the function $Z_{k + 1}(x, \la)$ has the zeros $\{ \la_{l,k,0} \}_{l \ge 1}$ and the poles $\{ \la_{l,k+1,0} \}_{l \ge 1}$. 
Denote
$$
d_j(\la) := \prod_{l = 1}^{\infty} \left( 1 - \frac{\la}{\la_{l,j,0}}\right), \quad j = \overline{1,n-1}, \quad
d_n(\la) := 1.
$$
(For simplicity, we assume that $\la_{l,j,0} \ne 0$. The opposite case requires minor changes). Thus, the function
$$
\mathcal G_{k+1}(x, \la) := Z_{k+1}(x, \la) \frac{d_{k+1}(\la)}{d_k(\la)}
$$
is entire in $\la$.

Since $\{ \la_{l,j,0} \}$ satisfy the asymptotics \eqref{asymptla}, then the asymptotic properties of the function $d_j(\la)$ are analogous to $\Delta_{j,j}(\la)$, which is the characteristic function of the boundary value problem $\mathcal L_j$ with the boundary conditions \eqref{bc}. Consequently, using the methods of \cite{Bond22-asympt}, one can show that 
\begin{equation} \label{asymptd}
d_j(\rho^n) \asymp \rho^{s_j - \frac{n(n-1)}{2}} \exp(\rho (\om_{j+1} + \om_{j+2} + \dots + \om_n)), \quad \rho \in \overline{\Gamma}_{s, \rho^*, \de},
\end{equation}
where the notation $f(\rho) \asymp g(\rho)$ means
$$
    C_1 |g(\rho)| \le |f(\rho)| \le C_2 |g(\rho)|, \quad C_1, C_2 > 0,
$$
and $s_j$ is the sum of all the orders in the boundary conditions \eqref{bc}:
$$
s_j := \frac{j(j-1)}{2} + \frac{(n-j)(n-j-1)}{2}.
$$

Using the estimates \eqref{estZ2} and \eqref{asymptd}, we get
$$
|\mathcal G_{k+1}(x, \la)| \le C |\rho|^{-(n-k)} |\exp(\rho \om_{k+1}(x-1))|, \quad \la = \rho^n, \, \rho \in \overline{\Gamma}_s, \, |\rho| \ge \rho^*.
$$
Since $k \ge n-p$, then $\exp(\rho \om_{k+1}(x-1))$ is bounded as $|\rho| \to \infty$. Hence $\mathcal G_{k+1}(x, \la) \to 0$ as $|\la| \to \infty$. By Liouville's Theorem, $\mathcal G_{k+1}(x, \la) \equiv 0$ and so $Z_{k+1}(x, \la) \equiv 0$. Consequently, it follows from \eqref{resZ}, \eqref{relZz}, and the assumption $\be_{l,k+1}\ne 0$ that
\begin{align*}
z_{l,k,1}(x) & = Z_{k+1}(x, \la_{l,k,1}) = 0, \quad k < n, \\
z_{l,k+1,0}(x) & = \frac{1}{\be_{l,k+1,0}} Z_{k+1, \langle -1 \rangle}(x, \la_{l,k+1,0}) = 0, \quad k < n-1.
\end{align*}

By induction, this implies the assertion of the lemma.
\end{proof}

\begin{lem} \label{lem:Zdown}
If $z_{l,p,0}(x) = 0$ for all $l \ge 1$, then $z_{l,k,\eps}(x) = 0$ for $(l,k,\eps) \in V$, $k = \overline{1,p-1}$.
\end{lem}

\begin{proof}
Suppose that $z_{l,k,0}(x) = 0$ for some $k \in \{ 2, 3,\dots, p\}$ and all $l \ge 1$. Then, it follows from Lemma~\ref{lem:resZ} that $Z_k(x, \la)$ is entire. On the other hand, the estimate \eqref{estZ2} implies that $Z_k(x, \la)$ is bounded in the whole $\la$-plane. Hence $Z_k(x, \la) \equiv 0$, so the relation \eqref{relZz} implies $z_{l,k-1,\eps}(x) = 0$, $l\ge 1$, $\eps \in \{ 0, 1 \}$. Induction yields the assertion of the lemma.
\end{proof}

Introduce the auxiliary functions
\begin{equation} \label{defB}
B_j(x, \la) := Z_j(x, \la) \overline{Z_{n-j+1}(x, (-1)^n\overline{\la})}, \quad j = \overline{1,n}.
\end{equation}

\begin{lem} \label{lem:intB}
There exists a sequence of circles $\left\{ \la \in \mathbb C \colon |\la| = \Theta_v \right\}$ with radii $\Theta_v \to \infty$ such that 
\begin{equation} \label{intB}
\lim_{v \to \infty} \oint\limits_{|\la| = \Theta_v} B_j(x, \la) \, d\la = 0, \quad j = \overline{1,n}.
\end{equation}
\end{lem}

\begin{proof}
Fix $j \in \{ 1, 2, \dots, n \}$. The estimate \eqref{estZ} implies
\begin{equation} \label{estZZ}
|Z_j(x, \rho^n) \overline{Z_{n-j+1}(x, (-1)^n\overline{\rho^n})}| \le C |\rho|^{-(n-1)} \left( \sum_{l,k} \frac{\xi_l}{|\rho - c_k l| + 1}\right)^2, \quad \rho \in \overline{\Gamma}_{s, \rho^*, \de}.
\end{equation}

Choose such radii $\theta_r \to \infty$ that
$$
\{ \rho \in \mathbb C \colon |\rho| = \theta_r \} \subset \overline{\Gamma}_{s,\rho^*,\de}, \quad \theta_{r + 1} - \theta_r > 1, \quad r \ge 1,
$$
for $s = 1, 2$, $\rho^*$, and $\de$, for which the estimate \eqref{estZZ} holds. Denote by $n_{r,k}$ the closest integer to $\frac{\theta_r}{|c_k|}$. Then
$$
|B_j(x, \rho^n)| \le C |\rho|^{-(n-1)} \left( \sum_{l,k} \frac{\xi_l}{|n_{r,k} - l| + 1}\right)^2, \quad |\rho| = \theta_r.
$$

For simplicity, suppose that $\{ \xi_l \} \in l_1$. Then
\begin{align*}
|B_j(x, \rho^n)| & \le C |\rho|^{-(n-1)} \sum_l (\sqrt{\xi_l})^2 \sum_{l,k} \frac{(\sqrt{\xi_l})^2}{(|n_{r,k} - l| + 1)^2} \\
& \le C |\rho|^{-(n-1)} \sum_{k = 1}^{n-1} g_{r,k}, \quad |\rho| = \theta_r,
\end{align*}
where
$$
g_{r,k} := \sum_{l = 1}^{\infty} \frac{\xi_l}{(|n_{r,k} - l| + 1)^2}.
$$

Clearly,
$$
\sum_{r = 1}^{\infty} g_{r,k} \le \sum_{l= 1}^{\infty} \xi_l \sum_{r = 1}^{\infty} \frac{1}{(|n_r- l|+ 1)^2} \le
C \sum_{u = 1}^{\infty} \frac{1}{u^2} < \infty, \quad k = \overline{1,n-1}.
$$

Hence, for each fixed $k \in \{ 1, 2, \dots, n-1 \}$, we have $\{ g_{r,k} \}_{r \ge 1} \in l_1$. Therefore, one can choose a subsequence $\{ r_v \}_{v \ge 1}$ such that $g_{r_v, k} = o(r_v^{-1})$ as $v \to \infty$. This implies $g_{r_v,k} = o(\theta_{r_v}^{-1})$, $v \to \infty$. Put $\Theta_v := \theta_{r_v}^n$. Then $B_j(x, \la) = o(|\la|^{-1})$ for $|\la| = \Theta_v$, $v \to \infty$. This yields the claim for the case $\{ \xi_l \} \in l_1$.

If $\{ \xi_l \} \not\in l_1$, then one can apply the technique of \cite{Bond21-tamkang} to show that 
$$
\{ z_{l,k,\eps}(x) w_{l,k}^{-1}(x) \} \in l_2, \quad 
\{ (z_{l,k,0}(x) - z_{l,k,1}(x)) w_{l,k}^{-1}(x) \} \in l_1.
$$

Using these relations, one can derive the estimate
$$
|B_j(x, \rho^n)| \le C |\rho|^{-(n-1)} \left( \sum_{l,k} \frac{\kappa_l}{|n_{r,k} - l| + 1}\right)^2, \quad |\rho| = \theta_r,
$$
with some sequence $\{ \kappa_l \} \in l_1$. Then, the proof of the lemma can be completed analogously to the case $\{ \xi_l \} \in l_1$.
\end{proof}

\begin{lem} \label{lem:resB}
The following relation holds:
$$
\Res_{\la = \la_{l,j,0}} B_j(x, \la) = \Res_{\la = \la_{l,j,0}} B_{j+1}(x, \la) = \be_{l,j,0} z_{l,j,0}(x) \overline{z_{l,n-j,0}(x)}, \quad j = \overline{1,n-1}.
$$
At all the other points, the functions $B_j(x, \la)$ are analytical in $\la$ for each fixed $x \in [0,1]$.
\end{lem}

\begin{proof}
The assertion of the lemma immediately follows from \eqref{sasd}, \eqref{relZz}, \eqref{defB}, and Lemma~\ref{lem:resZ}.
\end{proof}

Proceed to the proof of Theorem~\ref{thm:inverse}. Consider two cases.

\smallskip

\textit{Case} $n = 2p$. Introduce the function
$$
B(x, \la) := \sum_{j = 1}^p (-1)^j B_j(x, \la).
$$

By Lemma~\ref{lem:intB},
\begin{equation} \label{intB2}
\lim_{v \to \infty} \oint\limits_{|\la| =\Theta_v} B(x, \la) \, d\la = 0.
\end{equation}

Lemma~\ref{lem:resB} implies that $B(x, \la)$ has the only poles $\{ \la_{l,p,0} \}_{l \ge 1}$ and
$$
\Res_{\la = \la_{l,p,0}} B(x, \la) = (-1)^p \be_{l,p,0} z_{l,p,0}(x) \overline{z_{l,p,0}(x)}.
$$
Therefore, calculating the integrals in \eqref{intB2} by the Residue Theorem, we obtain
$$
\sum_{l = 1}^{\infty} \be_{l,p,0} |z_{l,p,0}(x)|^2 = 0.
$$

By the hypothesis of Theorem~\ref{thm:inverse}, we have $(-1)^{p + 1}\be_{l,p,0} > 0$. This implies $z_{l,p,0}(x) = 0$ for all $l \ge 1$. Applying Lemmas~\ref{lem:Zup} and \ref{lem:Zdown}, we conclude that $z_{l,k,\eps}(x) = 0$ for all $(l,k,\eps) \in V$.

\smallskip

\textit{Case} $n = 2p + 1$. Calculating the integral in \eqref{intB} by the Residue Theorem and using Lemma~\ref{lem:resB}, we get by induction for $j = \overline{1,p}$ that
\begin{equation} \label{relsum}
\sum_{l = 1}^{\infty} \Res_{\la = \la_{l,j,0}} B_j(x,\la) = \sum_{l = 1}^{\infty} \be_{l,j,0} z_{l,j,0}(x) \overline{z_{l,n-j,0}(x)} = 0.
\end{equation}

Consider the radii $\Theta_v$, $v \ge 1$, from Lemma~\ref{lem:intB}. Denote by $\Upsilon_v$ the boundary of the half-circle 
$$
\left\{ \la \in \mathbb C \colon |\la| < \Theta_v, \, (-1)^{p+1} \mbox{Re}\, \la > 0 \right\}.
$$

By virtue of Lemma~\ref{lem:resB}, the function $B_{p+1}(x, \la)$ has the poles $\{ \la_{l,p,0} \}$ and $\{ \la_{l,p+1,0} \}$. By the hypotheses of Theorem~\ref{thm:inverse}, we have $\la_{l,p+1,0} = -\overline{\la_{l,p,0}}$ and $(-1)^{p+1}\mbox{Re} \la_{l,p,0} > 0$, so $(-1)^{p+1}\mbox{Re} \la_{l,p+1,0} < 0$. Therefore, using the Residue Theorem, Lemma~\ref{lem:resB}, and \eqref{relsum}, we obtain
\begin{align*}
\lim_{v \to \infty} \frac{1}{2 \pi i} \oint\limits_{\Upsilon_v} B_{p+1}(x, \la) \, d\la & = \lim_{v \to \infty} \sum_{|\la_{l,p,0}| < \Theta_v} \Res_{\la = \la_{l,p,0}} B_{p+1}(x, \la) \\ & = \sum_{l = 1}^{\infty} \be_{l,p,0} z_{l,p,0}(x) \overline{z_{l,p+1,0}(x)} = 0.
\end{align*}

On the other hand, $\Upsilon_v = [-i \Theta_v, i \Theta_v] \cup \Upsilon_v^+$, where $\Upsilon_v^+$ is the arc $\{ |\la| = \Theta_v, \, 0 \le (-1)^{p+1} \arg \la \le \pi \}$, and
$$
\lim_{v \to \infty} \frac{1}{2\pi i} \int\limits_{\Upsilon_v^+} B_{p+1}(x, \la) \, d\la = 0.
$$
Consequently,
$$
\int_{-i\infty}^{i\infty} B_{p+1}(x, \la) \, d\la = 0.
$$
Using \eqref{defB} and putting $\la = i\tau$, we arrive at the relation
$$
\int_{-\infty}^{\infty} |Z_{p+1}(x, i\tau)|^2 \, d\tau = 0,
$$
which implies $Z_{p+1}(x, \la) \equiv 0$. The relations \eqref{relZz} and \eqref{resZ} imply $z_{l,p,\eps}(x) = 0$, $\eps = 0,1$, and $z_{l,p+1,0}(x)= 0$, respectively, for all $l \ge 1$. Using Lemmas~\ref{lem:Zup} and~\ref{lem:Zdown}, we conclude that $z_{l,k,\eps}(x) = 0$ for all $(l,k,\eps) \in V$.

Thus, in the both cases $n = 2p$ and $n = 2p+1$, we obtain $\zeta(x) = 0$, which finishes the proof of Theorem~\ref{thm:inverse}.
\end{proof}

\section{Inverse spectral problem} \label{sec:ip}

In this section, we apply Theorem~\ref{thm:inverse} to obtain necessary and sufficient conditions for solvability of an inverse spectral problem for equation \eqref{eqv} with the coefficients $\tau_{\nu} \in W_2^{\nu-1}[0,1]$, $\nu = \overline{0,n-2}$. In other words,
\begin{itemize}
    \item $\tau_0$ belongs to the space of generalized functions $W_2^{-1}[0,1]$, whose antiderivatives belong to $L_2[0,1]$.
    \item $\tau_1$ belongs to $W_2^0[0,1] = L_2[0,1]$.
    \item For $k \ge 1$, $\tau_{k+1}$ belongs to the Sobolev space $W_2^k[0,1]$ of functions $f(x)$ such that $f^{(k)} \in L_2[0,1]$.
\end{itemize}

On the one hand, one can reduce equation \eqref{eqv} to the form \eqref{eqvp},
where
\begin{equation} \label{ptau}
p_s = \sum_{k = \lceil s/2\rceil}^{\min \{s, \lfloor n/2\rfloor - 1\}} C_k^{s-k} \left(\tau_{2k}^{(2k-s)} + \tau_{2k+1}^{(2k-s+1)}\right) + \sum_{k = \lceil (s-1)/2\rceil}^{\min \{ s,\lfloor (n-1)/2 \rfloor\}-1} 2 C_k^{s-k-1} \tau_{2k+1}^{(2k+1-s)},
\end{equation}
$C_k^j := \frac{k!}{j!(k-j)!}$ are the binomial coefficients, the notations $\lfloor a \rfloor$ and $\lceil a \rceil$ mean the rounding of $a$ down and up, respectively, and $\tau_{n-1}=0$. Clearly, $p_s \in W_2^{s-1}[0,1]$, $s = \overline{0,n-2}$.

On the other hand, equation \eqref{eqv} can be transformed to the first-order system \eqref{sys}. 
For this purpose, we apply the results of \cite{Bond23-mmas}. In \cite{Bond23-mmas}, associated matrices $F(x)$ were constructed for differential expression of the following general form with various singularity orders $\{ i_{\nu} \}_{\nu = 0}^{n-2}$:
\begin{align} \nonumber
y^{(n)} & + \sum_{k = 0}^{\lfloor n/2\rfloor} (-1)^{i_{2k} + k} \left( \sigma_{2k}^{(i_{2k})}(x) y^{(k)}\right)^{(k)} \\ \label{deflsi}
& + \sum_{k = 0}^{\lfloor (n-1)/2 \rfloor-1} (-1)^{i_{2k+1} 
 + k + 1}\left( \bigl( \sigma_{2k+1}^{(i_{2k+1})}(x) y^{(k)}\bigr)^{(k+1)} + \bigl(\sigma_{2k+1}^{(i_{2k+1})}(x) y^{(k+1)}\bigr)^{(k)}\right), 
\end{align}
where $\{ \sigma_{\nu}(x) \}_{\nu = 0}^{n-2}$ are regular functions on $(0,1)$. For the differential expression $\ell_n(y)$ in \eqref{eqv} with $\tau_{\nu} \in W_2^{\nu-1}[0,1]$, one can put $i_0 := 1$, $\sigma_0 := -\tau_0^{(-1)}$, $i_{\nu} := 0$ and $\sigma_{\nu} := (-1)^{\lfloor \nu/2 \rfloor +\nu} \tau_{\nu}$ for $\nu \ge 1$. 
Then, according to the results of \cite[Section~2]{Bond23-mmas}, the associated matrix $F(x)$ can be obtained as follows.
Define the matrix function $Q(x) =[q_{k,j}(x)]_{k,j = 0}^p$, $p= \lfloor n/2 \rfloor$, by the relations
\begin{gather*}
q_{0,1} := \sigma_0 + \sigma_1, \quad q_{1,0} := \sigma_0- \sigma_1, \quad q_{k,k} := \sigma_{2k}, \quad k = \overline{1, p-1}, \\ q_{k,k+1} := \sigma_{2k+1}, \quad q_{k+1,k} := -\sigma_{2k+1}, \quad k = \overline{1,n-p-2}.
\end{gather*}
For $n \ge 3$, construct $F(x) = [f_{k,j}(x)]_{k,j = 1}^n$ by the formulas
\begin{equation} \label{fkj}
f_{k,j} := (-1)^{k+n+1} q_{j-1,n-k}, \quad k = \overline{p+1,n}, \: j = \overline{1,n-p}, \quad
f_{k,k+1} := 1, \quad k =\overline{1,n-1}.
\end{equation}
All the other entries of $Q(x)$ and $F(x)$ are assumed to be zero.
For example,
\begin{align*}
n = 4 \colon & \quad 
Q(x) = \begin{bmatrix}
            0 & \sigma_0 + \sigma_1 & 0 \\
            \sigma_0 - \sigma_1 & \sigma_2 & 0 \\
            0 & 0 & 0 
      \end{bmatrix},
\quad
F(x) = \begin{bmatrix}
            0 & 1 & 0 & 0 \\
            0 & 0 & 1 & 0 \\
            -(\sigma_0+\sigma_1) & -\sigma_2 & 0 & 1 \\
            0 & \sigma_0 - \sigma_1 & 0 & 0
       \end{bmatrix}, \\
n = 5 \colon & \quad
Q(x) = \begin{bmatrix}
            0 & \sigma_0 + \sigma_1 & 0 \\
            \sigma_0 - \sigma_1 & \sigma_2 & \sigma_3 \\
            0 & -\sigma_3 & 0 
      \end{bmatrix},
\quad
F(x) = \begin{bmatrix}
        0 & 1 & 0 & 0 & 0 \\
        0 & 0 & 1 & 0 & 0 \\
        0 & -\sigma_3 & 0 & 1 & 0 \\
        \sigma_0 +\sigma_1 & \sigma_2 & -\sigma_3 & 0 & 1 \\
        0 & -(\sigma_0 - \sigma_1) & 0 & 0 & 0
\end{bmatrix}.
\end{align*}
The construction for $n = 2$ is different, see \eqref{F2}. 

Obviously, $\sigma_{\nu} \in L_2[0,1]$ for $\nu \ge 0$, so $F \in \mathfrak F_n$. Define the quasi-derivatives $y^{[k]}$ and the domain $\mathcal D_F$ by using \eqref{quasi} and \eqref{defDF}, respectively. Then, by Theorem~2.2 in \cite{Bond23-mmas}, for any $y \in \mathcal D_F$, the relation $\ell_n(y) = y^{[n]}$ holds, which implies \textit{the regularization} of the differential expression $\ell_n(y)$. Thus, equation \eqref{eqv} can be represented as \eqref{eqvn} or as the first-order system \eqref{sys}. Note that there are different ways to choose an associated matrix $F(x)$ for regularization of the differential expression $\ell_n(y)$. In particular, one can use the regularization of Mirzoev and Shkalikov \cite{MS16, MS19} or choose other singularity orders $i_0 \ge 1$, $i_{\nu} \ge 0$, $\nu \ge 1$ to represent $\ell_n(y)$ in the form \eqref{deflsi}. Anyway, the choice of an associated matrix does not influence on the spectral data $\{ \la_{l,k}, \mathcal N(\la_{l,k}) \}_{l \ge 1, \,k = \overline{1,n-1}}$ (see \cite{Bond23-reg}). For definiteness, we use the associated matrix constructed by formulas \eqref{fkj}.

Consider the boundary value problems $\mathcal L_k$, $k = \overline{1,n-1}$, for equation \eqref{eqv} with the boundary conditions \eqref{bc}. 
We will write that $\{ \tau_{\nu} \}_{\nu = 0}^{n-2} \in W_{simp}$ if $\tau_{\nu} \in W_2^{\nu-1}[0,1]$, $\nu = \overline{0,n-2}$, and the corresponding eigenvalues $\{ \la_{l,k} \}_{l \ge 1, \, k = \overline{1,n-1}}$ satisfy (A-1) and (A-2). Consider the following inverse spectral problem.

\begin{ip} \label{ip:tau}
Given the spectral data $\{ \la_{l,k}, \be_{l,k} \}_{l \ge 1, \, k = \overline{1,n-1}}$, find the coefficients $\{ \tau_{\nu} \}_{\nu = 0}^{n-2} \in W_{simp}$.
\end{ip}

It has been proved in \cite{Bond22-alg, Bond23-loc} that, under the assumptions (A-1) and (A-2), the spectral data $\{ \la_{l,k}, \be_{l,k} \}_{l \ge 1, \, k= \overline{1,n-1}}$ for equation \eqref{eqvp} uniquely specify the coefficients $p_s \in W_2^{s-1}[0,1]$, $s = \overline{0,n-2}$. Since the relation \eqref{ptau} implies a bijection between $\{ p_s \}_{s = 0}^{n-2}$ and $\{ \tau_{\nu} \}_{\nu = 0}^{n-2}$ in the corresponding functional spaces, we immediately arrive at the following uniqueness proposition for Inverse Problem~\ref{ip:tau}.

\begin{prop} \label{prop:uniq}
The spectral data $\{ \la_{l,k}, \be_{l,k} \}_{l \ge 1, \, k = \overline{1,n-1}}$ uniquely determine the coefficients $\tau_{\nu} \in W_{simp}$.
\end{prop}

Moreover, Theorem~2 in \cite{Bond23-loc} implies the following sufficient conditions for the existence of solution for Inverse Problem~\ref{ip:tau}.

\begin{prop} \label{prop:suff}
Let complex numbers $\{ \la_{l,k}, \be_{l,k} \}_{l \ge 1, \, k = \overline{1,n-1}}$ satisfy (A-1), (A-2), and $\be_{l,k} \ne 0$ for all $l,k$. Suppose that there exists a model problem with  coefficients $\{ \tilde \tau_{\nu} \}_{\nu = 0}^{n-2} \in W_{simp}$ such that:
\begin{enumerate}
\item $\{ l^{n-2} \xi_l \}_{l \ge 1} \in l_2$, where the numbers $\xi_l$ were defined in \eqref{defxi}. 
\item The operator $(I - \tilde R(x))$, which is constructed by using $\{ \la_{l,k}, \be_{l,k} \}_{l \ge 1, \, k= \overline{1,n-1}}$ and the model problem according to Section~\ref{sec:main}, has a bounded inverse operator for each fixed $x \in [0,1]$. 
\end{enumerate}

Then, there exists a unique solution $\{ \tau_{\nu} \}_{\nu = 0}^{n-2} \in W_{simp}$ of Inverse Problem~\ref{ip:tau} for the data $\{ \la_{l,k}, \be_{l,k} \}_{l \ge 1, \, k = \overline{1,n-1}}$.
\end{prop}

A disadvantage of Proposition~\ref{prop:suff} is the requirement of the existence for the bounded operator $(I - \tilde R(x))^{-1}$. In general, it is difficult to verify this condition. However, in the self-adjoint case, we can apply Theorem~\ref{thm:inverse} for this purpose.

Denote by $W_{simp}^+$ the class of coefficients $\{\tau_{\nu} \}_{\nu = 0}^{n-2} \in W_{simp}$ such that the functions $i^{n+\nu} \tau_{\nu}$ are real-valued for $\nu = \overline{0,n-2}$.  Then, the associated matrix $F(x)$, which is constructed by formulas \eqref{fkj}, belongs to $\mathfrak F_{n,simp}^+$. Therefore, combining Theorem~\ref{thm:inverse} and Proposition~\ref{prop:suff}, we arrive at the following result.

\begin{thm} \label{thm:nsc}
Let complex numbers $\{ \la_{l,k}, \be_{l,k} \}_{l \ge 1, \, k = \overline{1,n-1}}$ satisfy (A-1), (A-2), \eqref{sasd}, \eqref{addhyp}, and $\be_{l,k} \ne 0$ for all $l, k$. Suppose that there exists a model problem with coefficients $\{\tilde \tau_{\nu}\}_{\nu = 0}^{n-2} \in W_{simp}^+$ such that $\{ l^{n-2} \xi_l \}_{l \ge 1} \in l_2$. Then, there exists a unique solution $\{ \tau_{\nu} \}_{\nu = 0}^{n-2}$ of Inverse Problem~\ref{ip:tau} for the data $\{ \la_{l,k}, \be_{l,k} \}_{l\ge 1, \, k = \overline{1,n-1}}$. Moreover, $\{ \tau_{\nu} \}_{\nu = 0}^{n-2} \in W_{simp}^+$.
\end{thm}

\begin{remark}
For even $n$, the conditions of Theorem~\ref{thm:nsc} are necessary and sufficient. Indeed, by necessity, the condition \eqref{addhyp} for even $n$ holds by necessity by virtue of Lemma~\ref{lem:positive} and a model problem can be chosen as $\tilde \tau_{\nu} := \tau_{\nu}$, $\nu = \overline{0,n-2}$. For odd $n$, the only ``gap'' between necessary and sufficient conditions is the requirement $(-1)^{p+1} \mbox{Re} \, \la_{l,p} > 0$, which plays an important role in the proof of Theorem~\ref{thm:inverse}.
\end{remark}

Note that the assumption $\{ l^{n-2} \xi_l \} \in l_2$ implies the asymptotics \eqref{asymptla} and \eqref{asymptbe} for $\{ \la_{l,k}, \be_{l,k} \}_{l \ge 1, \, k= \overline{1,n-1}}$, because the similar asymptotics hold for $\{ \tilde \la_{l,k}, \tilde \be_{l,k} \}_{l \ge 1, \, k = \overline{1,n-1}}$. However, the condition $\{ l^{n-2} \xi_l \} \in l_2$ is more strong. In order to achieve it, one has to find a model problem with the coefficients $\tilde c_{j,k} = c_{j,k}$ and $\tilde d_{j,k} = d_{j,k}$ in the sharp asymptotics
\begin{align*}
\la_{l,k} & = l^n \left(c_{0,k} + c_{1,k} l^{-1} + c_{2,k} l^{-2} + \dots + c_{n-1,k} l^{-(n-1)} + l^{-(n-1)} \varkappa_{l,k} \right), \\
\be_{l,k} & = -\la_{l,k} \left(1 + d_{1,k} l^{-1} + d_{2,k} l^{-2} + \dots + d_{n-2,k} l^{-(n-2)} + l^{-(n-2)} \varkappa_{l,k}^0\right),
\end{align*}
where $\{\varkappa_{l,k} \}, \{ \varkappa_{l,k}^0 \} \in l_2$. This task is explicitly solved for $n = 2, 3, 4$ in the next section. But for higher orders, it becomes very technically complicated.

\section{Examples} \label{sec:ex}

In this section, we consider Inverse Problem~\ref{ip:tau} for $n = 2, 3, 4$ and $\{ \tau_{\nu} \}_{\nu = 0}^{n-2} \in W_{simp}^+$. We obtain the corollaries of Theorem~\ref{thm:nsc} on the spectral data characterization for these cases. For $n = 2$ and $n = 3$, our results coincide with the results of \cite{HM03} and \cite{Bond23-res}, respectively. For $n = 4$, our result (Theorem~\ref{thm:nsc4}) is novel.

\subsection{Second order}

For $n = 2$, equation \eqref{eqv} turns into the Sturm-Liouville equation
\begin{equation} \label{StL2}
y'' + \tau_0 y = \la y, \quad x \in (0, 1),
\end{equation}
where $\tau_0$ is a real-valued potential of $W_2^{-1}[0,1]$. Then, we have the only problem $\mathcal L_1$ with the Dirichlet boundary conditions 
\begin{equation} \label{dir}
y(0) = y(1) = 0.
\end{equation}
It is well-known (see, e.g., \cite{HM03}) that the corresponding eigenvalues $\la_{l,1} =: \la_l$ are real and simple. Furthermore, $\be_l := \be_{l,1} = \left( \int_0^1 y_l^2(x) \, dx \right)^{-1}$, where $\{ y_l(x) \}_{l \ge 1}$ are the eigenfunctions of the problem $\mathcal L_1$, normalized by the condition $y_l^{[1]}(0) = 1$. The asymptotics \eqref{asymptla} and \eqref{asymptbe} take the form
\begin{equation} \label{asympt2}
\la_l = -(\pi l + \varkappa_l)^2, \quad \be_l = 2 (\pi l)^2 (1 + \varkappa_l^0), \quad l \ge 1, \quad \{ \varkappa_l \}, \, \{\varkappa_l^0 \} \in l_2.
\end{equation}
Therefore, choosing any real-valued model potential $\tilde \tau_0 \in W_2^{-1}[0,1]$, we get $\{ \tilde \tau_0 \} \in W_{simp}^+$ and $\{ \xi_l \} \in l_2$. Hence, Theorem~\ref{thm:nsc} implies the following corollary, which is equivalent to the spectral data characterization in \cite{HM03}.

\begin{cor}
For numbers $\{ \la_l, \be_l \}_{l \ge 1}$ to be the spectral data of the Sturm-Liouville problem \eqref{StL2}--\eqref{dir} with a real-valued potential $\tau_0 \in W_2^{-1}[0,1]$, it is necessary and sufficient to satisfy the asymptotics \eqref{asympt2} and the conditions
$$
\la_l \in \mathbb R, \quad \la_l \ne \la_{l_0} \: (l \ne l_0), \quad \be_l > 0, \quad l \ge 1.
$$
\end{cor}

\subsection{Third order}

For $n = 3$, equation \eqref{eqv} takes the form
$$
y''' + (\tau_1 y)' + \tau_1 y' + \tau_0 y = \la y, \quad x \in (0,1),
$$
where the functions $i \tau_0(x)$ and $\tau_1(x)$ are real-valued, $\tau_0 \in W_2^{-1}[0,1]$, $\tau_1 \in L_2[0,1]$. Then, we have two boundary value problems:
\begin{align*}
\mathcal L_1 \colon & \quad y(0) = 0, \quad y(1) = y'(1) =0, \\
\mathcal L_2 \colon & \quad y(0)=y'(0) = 0, \quad y(1) = 0,
\end{align*}
and the corresponding spectral data satisfy the following asymptotics (see \cite{Bond23-res}):
\begin{equation} \label{asymptla3}
        \la_{l,k} = (-1)^{k+1} \left( \frac{2\pi}{\sqrt 3} \Bigl( l + \frac{1}{6} - \frac{\theta}{2\pi^2 n} + \frac{\varkappa_l}{l}  \Bigr)\right)^3, \quad \beta_{l,k} = -3 \la_{l,k} \left( 1 + \frac{\varkappa_l^0}{l}\right), 
\end{equation}
where $l \ge 1$, $k = 1, 2$,
$\theta = \int\limits_0^1 \tau_1(x) \, dx$, $\{ \varkappa_l \}, \{ \varkappa_l^0 \} \in l_2$. Obviously, the coefficient $\theta$ can be found from the eigenvalue asymptotics. By choosing a model problem $\{ \tilde \tau_0, \tilde \tau_1 \} \in W_{simp}^+$ with $\int_0^1 \tilde \tau_1(x) \, dx = \theta$, we achieve $\{ l \xi_l \}_{l \ge 1} \in l_2$. Consequently, Theorem~\ref{thm:nsc} implies the following corollary, which is a special case of Theorem~2.5 in \cite{Bond23-res}.

\begin{cor}
Let complex numbers $\{ \la_{l,k}, \be_{l,k} \}_{l \ge 1, \, k =1, 2}$ satisfy the assumptions (A-1), (A-2), $\la_{l,1}= -\overline{\la_{l,2}}$, $\be_{l,1} = -\overline{\be_{l,2}}$, $\mbox{Re}\, \la_{l,1} > 0$, $\be_{l,1} \ne 0$ for $l \ge 1$ and the asymptotics \eqref{asymptla3} with a real coefficient $\theta$. Then, there exists a unique solution $\{\tau_0, \tau_1 \}$ of Inverse Problem~\ref{ip:tau} with the spectral data $\{ \la_{l,k}, \be_{l,k} \}_{l \ge 1, \, k =1, 2}$.
\end{cor}

Note that in \cite{Bond23-res} the inverse problem has been investigated in a more general form, when the assumption (A-2) can be violated.

\subsection{Fourth order}

Consider equation \eqref{eqv} for $n = 4$:
\begin{equation} \label{eqv4}
y^{(4)} + (\tau_2(x) y')' + (\tau_1(x) y)' + \tau_1(x) y' + \tau_0(x) y = \la y(x), \quad x \in (0, 1),
\end{equation}
where $\{ \tau_0, \tau_1, \tau_2 \} \in W_{simp}^+$. This means
$\tau_0 \in W_2^{-1}[0,1]$, $\tau_1 \in L_2[0,1]$, $\tau_2 \in W_2^1[0,1]$, the functions $\tau_0$, $i \tau_1$, and $\tau_2$ are real-valued. The spectral data $\{ \la_{l,k}, \be_{l,k} \}_{l \ge 1, \, k = 1, 2, 3}$ are associated with the boundary value problems $\mathcal L_k$, $k = 1, 2, 3$, for equation \eqref{eqv4} with the following boundary conditions:
\begin{align*} 
    \mathcal L_1 \colon & \quad y(0) = 0, \quad y(1) = y'(1) = y''(1) = 0, \\
    \mathcal L_2 \colon & \quad y(0) = y'(0) = 0, \quad y(1) = y'(1) = 0, \\ 
    \mathcal L_3 \colon & \quad y(0) = y'(0) = y''(0) = 0, \quad y(1) = 0.    
\end{align*}

Theorem~\ref{thm:nsc} and the results of \cite{Bond23-asympt4} together imply the following theorem on the spectral data characterization for the fourth-order equation \eqref{eqv4}.

\begin{thm} \label{thm:nsc4}
For complex numbers $\{ \la_{l,k}, \be_{l,k} \}_{l \ge 1, \, k = 1, 2, 3}$ to be the spectral data of $\{ \tau_0, \tau_1, \tau_2 \} \in W_{simp}^+$, it is necessary and sufficient to fulfill the conditions (A-1), (A-2), \eqref{sasd}, $\be_{l,2} < 0$ and $\be_{l,2\pm 1} \ne 0$ for all $l \ge 1$, and the asymptotic relations
\begin{align} \label{asymptla13}
\la_{l,2\pm 1} = - \biggl( & \Bigl( \sqrt 2 \pi l + \frac{\pi}{2 \sqrt 2}\Bigr)^4 - \theta \Bigl( \sqrt 2 \pi l + \frac{\pi}{2 \sqrt 2}\Bigr)^2  - \frac{t_0 + t_1 \mp 4\sigma}{\sqrt 2} \Bigl( \sqrt 2 \pi l + \frac{\pi}{2 \sqrt 2}\Bigr) + l \varkappa_{l,2\pm 1} \biggr), \\ \label{asymptla2}
\la_{l,2} = \Bigl( \pi l & + \frac{\pi}{2} \Bigr)^4 - \theta \Bigl( \pi l + \frac{\pi}{2} \Bigr)^2 + (t_0 + t_1) \Bigl( \pi l + \frac{\pi}{2} \Bigr) + l \varkappa_{l,2}, \\ \label{asymptbe4}
\be_{l,2\pm 1} = -4 & \la_{l,2\pm 1} \left( 1 + \frac{t_0 + \theta}{ 8(\pi l)^2}  + \frac{\varkappa_{l,2 \pm 1}^0}{l^2}\right), \quad
\be_{l,2} = -4 \la_{l,2} \left( 1 + \frac{t_0 + 2 \theta}{4 (\pi l)^2} + \frac{\varkappa_{l,2}^0}{l^2}\right),
\end{align}
where $\{ \varkappa_{l,k} \}, \, \{ \varkappa_{l,k}^0 \} \in l_2$ and
$$
\theta = \int_0^1 \tau_2(x) \, dx, \quad t_0 = \tau_2(0), \quad t_1 = \tau_2(1), \quad \sigma = \int_0^1 \tau_1(x) \, dx.
$$
\end{thm}

\begin{proof}
By necessity, the asymptotic relations \eqref{asymptla13}--\eqref{asymptbe4} have been obtained in \cite{Bond23-asympt4} and the conditions \eqref{sasd}, $\be_{l,2} < 0$, and $\be_{l,2\pm1} \ne 0$ follow from Lemma~\ref{lem:simps}, Lemma~\ref{lem:positive}, and the structure of the weight matrices \eqref{structN}, respectively. 

By sufficiency, the asymptotics \eqref{asymptla13}--\eqref{asymptbe4} allow us to construct a model problem $\{ \tilde \tau_0, \tilde \tau_1, \tilde \tau_2 \}$ such that $\{ l^2 \xi_l \} \in l_2$. Indeed, one can successively find the constants 
\begin{align*}
\theta & := \lim_{l \to \infty} \left( \Bigl( \pi l + \frac{\pi}{2} \Bigr)^2 - \la_{l,2} \Bigl( \pi l + \frac{\pi}{2} \Bigr)^{-2}\right), \\
t_0 & := \lim_{l \to \infty} 4 (\pi l)^2 (\be_{l,2} + 4 \la_{l,2}) - 2\theta, \\
t_1 & := \lim_{l \to \infty} \left( \la_{l,2} \Bigl( \pi l + \frac{\pi}{2}\Bigr)^{-1} - \Bigl( \pi l + \frac{\pi}{2}\Bigr)^3 + \theta \Bigl( \pi l + \frac{\pi}{2}\Bigr) \right) -t_0, \\
\sigma & := \lim_{l \to \infty} \frac{\la_{l,1} - \la_{l,3}}{8 \left( \pi l + \frac{\pi}{4} \right)},
\end{align*}
and construct the functions $\{ \tilde \tau_0, \tilde \tau_1, \tilde \tau_2 \}$ such that $\tilde \theta = \theta$, $\tilde t_0 = t_0$, $\tilde t_1 = t_1$, and $\tilde \sigma = \sigma$. For example, put
$$
\tilde \tau_0(x) \equiv 0, \quad \tilde \tau_1(x) \equiv \sigma, \quad \tilde \tau_2(x) = (3 t_0 +  3 t_1 - 6 \theta) x^2 + (-4 t_0 - 2 t_1 + 6 \theta) x + t_0.
$$
Then, the spectral data $\{ \tilde \la_{l,k}, \tilde \be_{l,k} \}_{l \ge 1, \, k = 1, 2,3}$ satisfy the asymptotics with the same main parts as \eqref{asymptla13}--\eqref{asymptbe4}. Hence, the sequences $\{ l^{-1} |\la_{l,k} - \tilde \la_{l,k}| \}$ and $\{ l^{-2} |\be_{l,k} - \tilde \be_{l,k}| \}$ belong to $l_2$. According to \eqref{defxi}, this immediately implies $\{ l^2 \xi_l \} \in l_2$. If $\{ \tilde \tau_0, \tilde \tau_1, \tilde \tau_2\} \not\in W_{simp}$, then one can perturb a finite number of the eigenvalues $\{ \tilde \la_{l,k} \}$ to achieve (A-1) and (A-2), and such perturbation does not influence on the asymptotics. Thus, for any data $\{ \la_{l,k}, \be_{l,k} \}_{l \ge 1, \, k = 1, 2, 3}$ satisfying the conditions of Theorem~\ref{thm:nsc4}, the hypothesis of Theorem~\ref{thm:nsc} is valid, so there exist $\{ \tau_0, \tau_1, \tau_2 \} \in W_{simp}^+$ with the spectral data $\{ \la_{l,k}, \be_{l,k} \}_{l \ge 1, \, k = 1, 2, 3}$. 
\end{proof}

\medskip

{\bf Funding.} This work was supported by Grant 21-71-10001 of the Russian Science Foundation, https://rscf.ru/en/project/21-71-10001/.

\noindent Natalia Pavlovna Bondarenko \\

Department of Mechanics and Mathematics, Saratov State University, \\
Astrakhanskaya 83, Saratov 410012, Russia, \\

\noindent e-mail: {\it bondarenkonp@info.sgu.ru}

\end{document}